\documentclass[11pt]{amsart}
\usepackage{amssymb,nicefrac,stmaryrd,mathtools}
\usepackage[all]{xy}
\usepackage{enumitem}
\usepackage[pdftitle={E-theory for C*-algebras over a space},
  pdfauthor={Marius Dadarlat and Ralf Meyer},
  pdfsubject={Mathematics; MSC 18E30, 19K35, 46L80, 55U35}
]{hyperref}

\usepackage[lite]{amsrefs}
\newcommand*{\MRref}[2]{ \href{http://www.ams.org/mathscinet-getitem?mr=#1}{MR \textbf{#1}}}
\newcommand*{\arxiv}[1]{\href{http://www.arxiv.org/abs/#1}{arXiv: #1}}

\usepackage{microtype}

\DeclareMathOperator{\Hom}{Hom}
\DeclareMathOperator{\Prim}{Prim}
\DeclareMathOperator{\Ext}{Ext}
\DeclareMathOperator{\PExt}{PExt}
\DeclareMathOperator{\Tor}{Tor}

\DeclareMathOperator{\Id}{Id}

\mathtoolsset{mathic}

\newcommand*{\KK}{\textup{KK}}
\newcommand*{\ima}{\textup i}

\newcommand*{\RE}{\mathcal R\textup{E}}
\newcommand*{\K}{\textup K}
\newcommand*{\Sphere}{\mathbb S}
\newcommand*{\Kcoef}{\underline\K}

\newcommand*{\Mat}{\mathbb M}

\newcommand*{\cp}{\textup{cp}}
\newcommand*{\bo}{\textup b}
\newcommand*{\Good}{\textup{Good}}
\newcommand*{\Ecat}{\mathfrak E}
\newcommand*{\Cuntza}{\mathcal O}

\newcommand*{\into}{\rightarrowtail}
\newcommand*{\prto}{\twoheadrightarrow}

\newcommand*{\cl}[1]{\overline{#1}}

\newcommand*{\primid}{\mathfrak p}
\newcommand*{\cov}{\mathcal U}
\newcommand*{\Primep}{\mathcal P}
\newcommand*{\Cont}{\textup C}
\newcommand*{\Mult}{\mathcal M}

\newcommand*{\Cstarcat}{\mathfrak{C^*alg}}
\newcommand*{\Cstarsep}{\mathfrak{C^*sep}}
\newcommand*{\KKcat}{\mathfrak{KK}}

\newcommand*{\id}{\textup{id}}
\newcommand*{\norm}[1]{\lVert#1\rVert}
\newcommand*{\Asm}[2][X]{\textup{Asymp}(#2)_{#1}}
\newcommand*{\asm}[2][X]{[\![#2]\!]_{#1}}

\newcommand*{\III}{\mathbb{I}}
\newcommand*{\E}{\textup E}
\newcommand*{\Comp}{\mathbb K}
\newcommand*{\C}{\mathbb C}
\newcommand*{\R}{\mathbb R}
\newcommand*{\Z}{\mathbb Z}
\newcommand*{\N}{\mathbb N}
\newcommand*{\Ideals}{\mathbb I}
\newcommand*{\Open}{\mathbb O}

\newcommand*{\nb}{\nobreakdash}

\newcommand*{\Bootstrap}{\mathcal B}
\newcommand*{\BootstrapE}{\mathcal B_\E}

\newcommand*{\Cst}{\textup C^*}

\newcommand*{\Star}{\texorpdfstring{$^*$\nb-}{*-}}

\newcommand*{\blank}{\text\textvisiblespace}

\newcommand*{\defeq}{\mathrel{\vcentcolon=}}

\numberwithin{equation}{section}

\theoremstyle{plain}
\newtheorem{theorem}[equation]{Theorem}
\newtheorem{proposition}[equation]{Proposition}
\newtheorem{lemma}[equation]{Lemma}

\newtheorem{corollary}[equation]{Corollary}
\theoremstyle{definition}
\newtheorem{definition}[equation]{Definition}

\theoremstyle{remark}
\newtheorem{remark}[equation]{Remark}

\newtheorem{example}[equation]{Example}

\begin{document}
\title[E-theory for $\Cst$\nb-algebras over topological spaces]{E-theory for $\Cst$\nb-algebras\\ over topological spaces}
\author{Marius Dadarlat}
\address{Department of Mathematics\\
  Purdue University\\
  150 N.~University Street\\
  West Lafayette, IN 47907-2067\\
  USA}
\email{mdd@math.purdue.edu}

\author{Ralf Meyer}
\address{Mathematisches Institut and Courant Research Centre ``Higher Order Structures''\\
  Georg-August Universit\"at G\"ottingen\\
  Bunsenstra{\ss}e 3--5\\
  37073 G\"ottingen\\
  Germany}
\email{meyerr@member.ams.org}
\begin{abstract}
  We define \(\E\)\nb-theory for separable \(\Cst\)\nb-algebras over second countable topological spaces and establish its basic properties.  This includes an approximation theorem that relates the \(\E\)\nb-theory over a general space to the \(\E\)\nb-theories over finite approximations to this space.  We obtain effective criteria for determining
  the invertibility of  \(\E\)\nb-theory elements over possibly infinite-dimensional spaces.
   Furthermore, we prove a Universal Multicoefficient Theorem for \(\Cst\)\nb-algebras over totally disconnected metrisable compact spaces.
\end{abstract}
\thanks{The first author was partially supported by NSF grant \#DMS--0801173. The second author was supported by the German Research Foundation (Deutsche Forschungsgemeinschaft (DFG)) through the Institutional Strategy of the University of G\"ottingen.}
\maketitle

\section{Introduction}
\label{sec:intro}

Eberhard Kirchberg~\cite{Kirchberg:Michael} proved a far-reaching classification theorem for non-simple, strongly purely infinite, stable, nuclear, separable \(\Cst\)\nb-algebras.  Roughly speaking, two such \(\Cst\)\nb-algebras are isomorphic once they have homeomorphic primitive ideal spaces --~call this space~\(X\)~-- and are \(\KK(X)\)-equivalent in a suitable bivariant \(\K\)\nb-theory for \(\Cst\)\nb-algebras over~\(X\).  To apply this classification theorem, we need tools to compute this bivariant \(\K\)\nb-theory.  Following Mikael R{\o}rdam~\cite{Rordam:Classification_extensions} and Alexander Bonkat~\cite{Bonkat:Thesis}, who dealt with the simplest non-trivial case, the non-Hausdorff space with two points, Universal Coefficient Theorems for \(\KK(X)\) have now been established over several finite spaces~\(X\) in~\cites{Restorff:Classification, Restorff:Thesis, Eilers-Restorff:Rordam_gen, Meyer-Nest:Filtrated_K}.  Here we concentrate on the special issues for infinite~\(X\).

Recall that Kasparov theory only satisfies excision for \(\Cst\)\nb-algebra extensions with a completely positive section.  Similar technical restrictions appear for all variants of Kasparov theory, including Kirchberg's.  This is a severe limitation.  For instance, excision does not hold in general for extensions of the form \(A(U) \into A \prto A/A(U)\) for an open subset~\(U\), where \(A(U)\) denotes the restriction of~\(A\) to~\(U\), extended by~\(0\) to a \(\Cst\)\nb-algebra over the original space, even if~\(A\) is nuclear.  In the non-equivariant case, such technical problems are resolved by passing to \(\E\)\nb-theory, which satisfies excision for all \(\Cst\)\nb-algebra extensions (see~\cite{Connes-Higson:Deformations}).  Here we define an analogue of \(\E\)\nb-theory for separable \(\Cst\)\nb-algebras over a second countable topological space~\(X\).  We establish that our new theory has the expected properties, including a universal property and exactness for all extensions of \(\Cst\)\nb-algebras over~\(X\).  If~\(X\) is a locally compact Hausdorff space, then our definitions agree with previous ones by Efton Park and Jody Trout in~\cite{Park-Trout:Representable_E} and by Radu Popescu in~\cite{Popescu:Equivariant_E}.  We also formulate sufficient criteria for the natural map \(\E_*(X;A,B)\to \KK_*(X;A,B)\) to be invertible.  For instance, this works if~\(X\) is locally compact and Hausdorff and~\(A\) is a continuous field of nuclear \(\Cst\)\nb-algebras over~\(X\).

Our definition of \(\E_*(X;A,B)\) is based on asymptotic homomorphisms satisfying an approximate equivariance condition.  An asymptotic homomorphism \(\varphi_t\colon A\to B\),
\(t\in [0,\infty)\), is called \emph{approximately \(X\)\nb-equivariant} if for each open subset \(U\subseteq X\), we have
\[
\lim_{t\to \infty} {}\norm{\varphi_t(a)}_{X\setminus U} =0
\qquad\text{for all \(a\in A(U)\),}
\]
where \(\norm{\varphi_t(a)}_{X\setminus U}\) denotes the norm of~\(\varphi_t(a)\) in the quotient \(B(X\setminus U) \defeq B/B(U)\) of~\(B\).

Let \(\cov=(U_n)_{n\in\N}\) be a countable basis for the topology of~\(X\).  For each \(n\in\N\), the open subsets \(U_1,\dotsc,U_n\) generate a finite topology~\(\tau_n\) on~\(X\).  Let~\(X_n\) be the \(\textup T_0\)\nb-quotient of \((X,\tau_n)\), this is a finite \(\textup T_0\)\nb-space.  The quotient map \(X\prto X_n\) allows us to view \(\Cst\)\nb-algebras over~\(X\) as \(\Cst\)\nb-algebras over~\(X_n\) for all \(n\in\N\).  Our first main result is a short exact sequence
\begin{equation}
  \label{eq:intro_approximation_sequence}
  \mathop{\varprojlim_{n\in\N}}\nolimits^1 \E_{*+1}(X_n;A,B) \into \E_*(X;A,B) \prto \varprojlim_{n\in\N} \E_*(X_n;A,B)
\end{equation}
for all separable \(\Cst\)\nb-algebras \(A\) and~\(B\) over~\(X\).  This is made plausible by the observation that an asymptotic homomorphism \(A\to B\) is approximately \(X\)\nb-equivariant if and only if it is approximately \(X_n\)\nb-equivariant for all \(n\in\N\).  Hence the space of approximately \(X\)\nb-equivariant asymptotic homomorphisms is the intersection of the spaces of approximately \(X_n\)\nb-equivariant asymptotic homomorphisms for \(n\in\N\).  Since there are, in general, technical problems with computing homotopy groups of intersections, we use a mapping telescope to establish the long exact sequence~\eqref{eq:intro_approximation_sequence}.

As an important application of~\eqref{eq:intro_approximation_sequence}, we give an effective criterion for invertibility of $\E$-theory elements: an element in \(\E_*(X;A,B)\) is invertible if and only if its image in \(\E_*\bigl(A(U),B(U)\bigr)\) is invertible for all \(U\in\Open(X)\).  As a consequence, if all two-sided closed ideals of a separable nuclear \(\Cst\)\nb-algebra~\(A\) with Hausdorff primitive spectrum~$X$ are \(\KK\)-contractible, then
\[
A\otimes \Cuntza_\infty \otimes \Comp\cong \Cont_0(X)\otimes \Cuntza_2 \otimes \Comp.
\]
This result solves the problem of characterising the trivial continuous fields with fibre $\Cuntza_2 \otimes \Comp$ within the class of strongly purely infinite, stable, continuous fields of \(\Cst\)\nb-algebras.  It is worth noting that in general the \(\KK\)-contractibility of ideals does not follow from the \(\KK\)-contractibility of the fibres.  Indeed, there are examples of separable nuclear continuous fields~$A$ over the Hilbert cube with all fibres isomorphic to~$\Cuntza_2$ and yet such that $\K_0(A)\neq 0$, see~\cite{Dadarlat:Fiberwise_KK}.

While~\eqref{eq:intro_approximation_sequence}, in principle, reduces the computation of \(\E_*(X;A,B)\) for infinite spaces~\(X\) to the corresponding problem for the finite approximations~\(X_n\), this does not yet lead to a Universal Coefficient Theorem.  If \(\E_*(X_n;A,B)\) is computable by Universal Coefficient Theorems for all \(n\in\N\), the latter will usually involve short exact sequences.  Thus we have to combine two short exact sequences, as in the computation of the \(\K\)\nb-theory for crossed products by~\(\Z^2\) using the Pimsner--Voiculescu exact sequence twice.  This can only be carried through if we have some extra information.  In terms of the general homological machinery  developed in~\cite{Meyer-Nest:Homology_in_KK}, we find that the homological dimension of \(\E\)\nb-theory over an infinite space~\(X\) may be one larger than the homological dimensions of the finite approximations~\(X_n\).  Thus it is usually~\(2\), which does not suffice for classification theorems.

In fact, it is well-known that filtrated \(\K\)\nb-theory cannot be a complete invariant for \(\Cst\)\nb-algebras over the one-point compactification of~\(\N\).  Here we observe that the counterexample in~\cite{Dadarlat-Eilers:Bockstein} may be transported easily to any compact Hausdorff space.

The good excision properties of \(\E\)\nb-theory are particularly useful to study the \(\E\)\nb-theoretic analogue of the bootstrap class.  For a finite space~\(X\), the bootstrap class for \(\KK(X)\) is studied in~\cite{Meyer-Nest:Bootstrap}.  When we replace \(\KK(X)\) by \(\E(X)\), the technical assumptions in~\cite{Meyer-Nest:Bootstrap} about completely positive sections disappear, so that a \(\Cst\)\nb-algebra~\(A\) over a finite space~\(X\) belongs to the \(\E\)\nb-theoretic bootstrap class over~\(X\) if and only if all the distinguished ideals \(A(U)\) for open subsets \(U\subseteq X\) belong to the usual non-equivariant \(\E\)\nb-theoretic bootstrap class.  As we shall see, the latter criterion becomes a useful \emph{definition} of the bootstrap class over an infinite space~\(X\).  In $\KK(X)$, this condition would not yet be sufficient for a reasonable definition of the bootstrap class.

If~\(X\) is the Cantor set or, more generally, a totally disconnected metrisable compact space, then we may resolve the counterexamples mentioned above by taking into account coefficients.  Our second main result is a Universal \emph{Multi}coefficient Theorem for \(\E_*(X;A,B)\) for  two \(\Cst\)\nb-algebras \(A\) and~\(B\) over~\(X\).  It assumes that \(A(U)\) belongs to the \(\E\)\nb-theoretic bootstrap class for all open subsets \(U\subseteq X\) and yields a natural exact sequence
\[
\Ext_{\Cont(X,\Lambda)}\bigl(\Kcoef(A)[1],\Kcoef(B)\bigr)
\into \E(X;A,B) \prto
\Hom_{\Cont(X,\Lambda)}\bigl(\Kcoef(A),\Kcoef(B)\bigr),
\]
where~\(\Kcoef\) denotes the \(\K\)\nb-theory of~\(A\) with coefficients, viewed as a countable module over the \(\Z/2\)\nb-graded ring \(\Cont(X,\Lambda)\) of locally constant functions from~\(X\) to the \(\Z/2\)\nb-graded ring~\(\Lambda\) of B\"ockstein operations (see~\cite{Dadarlat-Loring:Multicoefficient}).  As a consequence, two \(\Cst\)\nb-algebras \(A\) and~\(B\) in the \(\E\)\nb-theoretic bootstrap class over~\(X\) are \(\E(X)\)-equivalent if and only if \(\Kcoef(A)\) and \(\Kcoef(B)\) are isomorphic as \(\Cont(X,\Lambda)\)-modules.

\section{\texorpdfstring{$\E$}{E}-theory for \texorpdfstring{$\Cst$}{C*}-algebras over non-Hausdorff spaces}
\label{sec:E-theory}

We recall some definitions from~\cite{Meyer-Nest:Bootstrap} regarding \(\Cst\)\nb-algebras over possibly non-Hausdorff topological spaces and then introduce equivariant \(\E\)\nb-theory for them.  Following the approach of Alain Connes and Nigel Higson in~\cite{Connes-Higson:Deformations}, we first describe \(\E\)\nb-theory concretely using asymptotic morphisms, then abstractly using a universal property.  For a locally compact Hausdorff space~\(X\), our definition is equivalent to previous ones for \(\Cont_0(X)\)-algebras by Efton Park and Jody Trout in~\cite{Park-Trout:Representable_E} and by Radu Popescu in~\cite{Popescu:Equivariant_E}.

\subsection{\texorpdfstring{$\Cst$}{C*}-algebras over non-Hausdorff spaces}
\label{sec:Cstar_over_X}

Here we recall some basic definitions from~\cite{Meyer-Nest:Bootstrap}.

For a \(\Cst\)\nb-algebra~\(A\), let \(\Prim(A)\) denote its primitive ideal space, equipped with the hull--kernel topology, and let \(\Ideals(A)\) be the set of ideals in~\(A\), partially ordered by inclusion.  For a topological space~\(X\), let \(\Open(X)\) be the set of open subsets of~\(X\), partially ordered by inclusion.  Both \(\Ideals(A)\) and \(\Open(X)\) are complete lattices, that is, any subset has both an infimum and a supremum.  It is shown in \cite{Dixmier:Cstar-algebres}*{\S3.2} that there is a canonical lattice isomorphism
\begin{equation}
  \label{eq:Prim_open_Ideal}
  \Open\bigl(\Prim(A)\bigr) \cong \Ideals(A),
  \qquad
  U \mapsto \bigcap \{\primid: \primid\in\Prim(A)\setminus U\}.
\end{equation}

\begin{definition}
  \label{def:Cstar_over_X}
  Let~\(X\) be a topological space.

  A \emph{\(\Cst\)\nb-algebra over~\(X\)} is a \(\Cst\)\nb-algebra~\(A\) with a continuous map~\(\psi\) from \(\Prim(A)\) to~\(X\).

  For an open subset~\(U\) of~\(X\), we let \(A(U)\in\Ideals(A)\) be the ideal that corresponds to \(\psi^{-1}(U)\in \Open(\Prim A)\) under the isomorphism~\eqref{eq:Prim_open_Ideal}.

  For a closed subset~\(S\) of~\(X\), we let \(A(S)\defeq A/ A(X\setminus S)\).  For \(a\in A\), we write~\(\norm{a}_S\) for the norm of the image of~\(a\) in the quotient \(\Cst\)\nb-algebra~\(A(S)\).

  More generally, if \(S\subseteq X\) is locally closed, that is, \(S=U\setminus V\) with open subsets \(V\subseteq U\subseteq X\), then we let \(A(S)\defeq A(U)/ A(V)\).  This quotient is independent of the choice of the open sets \(U\) and~\(V\) with \(S=U\setminus V\).

  Let \(A\) and~\(B\) be \(\Cst\)\nb-algebras over~\(X\).  A \Star{}homomorphism \(f\colon A\to B\) is called \emph{\(X\)\nb-equivariant} or a \emph{\Star{}homomorphism over~\(X\)} if~\(f\) maps \(A(U)\) into \(B(U)\) for all open subsets~\(U\) of~\(X\).

  Let \(\Cstarcat(X)\) be the category whose objects are the \(\Cst\)\nb-algebras over~\(X\) and whose morphisms are the \Star{}homomorphisms over~\(X\).  Let \(\Cstarsep(X)\) be the full subcategory of \emph{separable} \(\Cst\)\nb-algebras over~\(X\) with \Star{}homomorphisms over~\(X\) as morphisms.
\end{definition}

We usually drop the map \(\Prim(A)\to X\) from our notation and simply call~\(A\) a \(\Cst\)\nb-algebra over~\(X\).

Although the above definition involves~\(X\), all that really matters is the lattice \(\Open(X)\).  It is explained in~\cite{Meyer-Nest:Bootstrap} that it is essentially no loss of generality to assume~\(X\) to be \emph{sober}.  In that case, we may recover~\(X\) from the lattice \(\Open(X)\) and the map \(\Prim(A)\to X\) from the map \(\Open(X)\to\Ideals(A)\), \(U\mapsto A(U)\) (see \cite{Meyer-Nest:Bootstrap}*{Lemma 2.25}), which may be any map that commutes with finite infima and arbitrary suprema.  Thus if~$X$ is a second countable, sober space, a \(\Cst\)\nb-algebra over~\(X\) is a \(\Cst\)\nb-algebra~\(A\) endowed with an order preserving map \(\Open(X)\to\Ideals(A)\), \(U\mapsto A(U)\), which satisfies the following conditions:
\begin{itemize}
\item[(1)] $A(\emptyset)=0$, $A(X)=A$,
\item[(2)] $A(U_1\cap U_2)=A(U_1)\cdot A(U_2)$,
\item[(3)] $A\bigl(\bigcup_{n=1}^\infty U_n\bigr)=\overline{\sum_{n=1}^\infty A(U_n)}$.
\end{itemize}
If a \(\Cst\)\nb-algebra~\(A\) satisfies the conditions (1) and~(2) and
\begin{itemize}
\item[(3')] $A(U_1\cup U_2)=A(U_1)+A(U_2)$,
\end{itemize}
then we say that $A$ is a \emph{quasi} \(\Cst\)\nb-algebra over~\(X\).  If~$B$ is a \(\Cst\)\nb-algebra over~\(X\) then $\Cont_\bo(T,B)$ and $\Cont_\bo(T,B) \bigm/ \Cont_0(T,B)$ for \(T\defeq [0,\infty)\) become quasi \(\Cst\)\nb-algebras over~\(X\), via the maps $U\mapsto \Cont_\bo(T,B(U))$ and
\[
U\mapsto \Cont_\bo(T,B(U))+\Cont_0(T,B)\bigm/ \Cont_0(T,B).
\]
However, they do not satisfy the condition~(3) above.

Let~\(X\) be a locally compact Hausdorff space and let~\(A\) be a \(\Cst\)\nb-algebra over~\(X\).  The continuous map \(\Prim(A)\to X\) induces a \Star{}homomorphism
\[
\Cont_\bo(X) \to \Cont_\bo\bigl(\Prim(A)\bigr) \cong Z\Mult(A),
\]
where \(Z\Mult(A)\) denotes the centre of the multiplier algebra of~\(A\).  One verifies that $\Cont_0(X)A$ is dense in~$A$, so that~\(A\) becomes a \(\Cont_0(X)\)-\(\Cst\)-algebra.  This yields an isomorphism of categories between \(\Cstarcat(X)\) and the category of \(\Cont_0(X)\)-\(\Cst\)-algebras with \(\Cont_0(X)\)-linear \Star{}homomorphisms as morphisms by \cite{Meyer-Nest:Bootstrap}*{Proposition 2.11}.

\subsection{Approximately equivariant asymptotic morphisms}
\label{sec:equiv_asm}

Recall:

\begin{definition}
  \label{def:asymptotic_morphism}
  An \emph{asymptotic morphism} between two \(\Cst\)\nb-algebras \(A\) and~\(B\) is a map \(\varphi\colon A\to \Cont_\bo(T,B)\) for \(T\defeq[0,\infty)\) that induces a \Star{}homomorphism
  \[
  \dot\varphi\colon A\to B_\infty \defeq \Cont_\bo(T,B) \bigm/ \Cont_0(T,B).
  \]
  The map~\(\varphi\) is equivalent to a family of maps \(\varphi_t\colon A\to B\) for \(t\in T\) such that \(t\mapsto \varphi_t(a)\) is a bounded continuous function from~\(T\) to~\(B\) for each \(a\in A\).  Such a family is an asymptotic morphism if and only if
  \[
  \varphi_t(a^*+\lambda b)-\varphi_t(a)^*-\lambda \varphi_t(b)\quad\text{and}\quad \varphi_t(a\cdot b) - \varphi_t(a)\cdot\varphi_t(b)
  \]
  converge to~\(0\) in the norm topology for \(t\to\infty\) for all \(a,b\in A\), \(\lambda\in\C\).

  Two asymptotic morphisms \(\varphi\) and~\(\varphi'\) are called \emph{equivalent} if \(\dot\varphi=\dot\varphi'\), that is, \(\varphi_t(a)-\varphi'_t(a)\) converges to~\(0\) for \(t\to\infty\) for all \(a\in A\).
\end{definition}

\begin{definition}
  \label{def:equivariant_asymptotic_morphism}
  An asymptotic morphism \(\varphi_t\colon A\to B\) from~\(A\) to~\(B\) is called \emph{approximately \(X\)\nb-equivariant} if, for any open subset \(U\subseteq X\),
  \begin{equation}
    \label{eq:def0_ah_X_equiv}
    \lim_{t\to\infty} {}\norm{\varphi_t(a)}_{X\setminus U} = 0
    \qquad\text{for all \(a\in A(U)\)}.
  \end{equation}

  Let \(\Asm{A,B}\) be the set of approximately \(X\)\nb-equivariant asymptotic morphisms \(A\to B\).
\end{definition}

Our definition of \(\Asm{A,B}\) requires \(X\)\nb-equivariance only in the limit, the individual maps~\(\varphi_t\) need not be \(X\)\nb-equivariant.

\begin{remark}
  \label{rem:approx_equivariant_equivalent}
If~\(\varphi\) is equivalent to an approximately \(X\)\nb-equivariant asymptotic morphism, then~\(\varphi\) itself is approximately \(X\)\nb-equivariant.
\end{remark}

\begin{lemma}
  \label{lem:X-equivariance}
  An asymptotic morphism~\(\varphi\) is approximately \(X\)\nb-equivariant if and only if, for all closed subsets~\(S\) of~\(X\),
  \[
  \limsup_{t\to\infty} {}\norm{\varphi_t(a)}_S \le \norm{a}_S
  \qquad\text{for all \(a\in A\).}
  \]
\end{lemma}

\begin{proof}
  Let \(U\defeq X\setminus S\).  The \(\limsup\)-criterion specialises to the definition of \(X\)\nb-equivariance for \(a\in A(U)\).  Conversely, for any \(\varepsilon>0\) we may split \(a\in A\) as \(a=a_1+a_2\) with \(a_1\in A(U)\) and \(\norm{a_2} < \norm{a}_S+ \varepsilon\) and estimate
  \[
  \limsup {}\norm{\varphi_t(a)}_S
  \le \limsup {}\norm{\varphi_t(a_1)}_S + \limsup {}\norm{\varphi_t(a_2)}.
  \]
  The \(X\)\nb-equivariance of~\(\varphi\) and \(a_1\in A(U)\) imply \(\lim {}\norm{\varphi_t(a_1)}_S=0\), and
  \[
  \limsup {}\norm{\varphi_t(a_2)} = \norm{\dot\varphi(a_2)} \le \norm{a_2} <\norm{a}_S+\varepsilon.
  \]
  Thus \(\limsup {}\norm{\varphi_t(a)}_S < \norm{a}_S+\varepsilon\) for all \(\varepsilon>0\).
\end{proof}

Let \(U\in\Open(X)\) and \(S\defeq X\setminus U\).  The quotient map \(\pi_S\colon B \to B(S)\) induces a map \(\tilde{\pi}_S\colon \Cont_\bo(T,B)\to \Cont_\bo\bigl(T, B(S)\bigr)\) whose kernel is \(\Cont_\bo(T,B(U))\).  Condition~\eqref{eq:def0_ah_X_equiv} is equivalent to
\begin{equation}
  \label{eq:def2_ah_X_equiv}
  \tilde{\pi}_S\circ\varphi\bigl(A(U)\bigr)
  \subseteq \Cont_0\bigl(T, B(S)\bigr).
\end{equation}

\begin{lemma}
  \label{lemma:on_X_equiv}
  An asymptotic morphism~\(\varphi\) is approximately \(X\)\nb-equivariant if and only if, for all open subsets~\(U\) of~\(X\),
  \begin{equation}
    \label{eq:def1_ah_X_equiv}
    \varphi\bigl(A(U)\bigr)\subseteq
    \Cont_\bo\bigl(T,B(U)\bigr)+\Cont_0(T,B).
  \end{equation}
\end{lemma}

\begin{proof}
  It is clear that~\eqref{eq:def1_ah_X_equiv} implies~\eqref{eq:def2_ah_X_equiv}.  To verify the converse, it suffices to prove
  \[
  (\tilde{\pi}_S)^{-1}\bigl(\Cont_0(T,B(S))\bigr) = \Cont_\bo\bigl(T,B(U)\bigr) + \Cont_0(T,B).
  \]
  The Bartle--Graves Theorem provides a continuous section \(\gamma\colon B(S)\to B\) of~\(\pi_S\).  Any \(f\in \Cont_\bo(T,B)\) decomposes as \(f=g+h\) with \(g\defeq f - \gamma\circ\tilde{\pi}_S(f)\) and \(h \defeq \gamma\circ \tilde{\pi}_S(f)\).  We have \(g\in \Cont_\bo\bigl(T,B(U)\bigr)\) and \(h\in \Cont_0(T,B)\) whenever \(\tilde{\pi}_S(f) \in \Cont_0\bigl(T,B(S)\bigr)\) because~\(\gamma\) is continuous.
\end{proof}

For Hausdorff spaces~\(X\), Park and Trout~\cite{Park-Trout:Representable_E} and Popescu~\cite{Popescu:Equivariant_E} defined an \(\E\)\nb-theory \(\RE_*(X;A,B)\) for \(\Cont_0(X)\)-algebras based on asymptotic morphisms~\(\varphi\) that are \emph{asymptotically \(\Cont_0(X)\)-equivariant} in the sense that \(\varphi(fa)-f\varphi(a)\in \Cont_0(T,B)\) for all \(a\in A\) and \(f\in \Cont_0(X)\); equivalently, \(\dot\varphi\colon A\to B_\infty\) is \(\Cont_0(X)\)-linear.

\begin{proposition}
  \label{pro:compare_Popescu}
  Let~\(X\) be a second countable locally compact Hausdorff space and let \(A\) and \(B\) be \(\Cont_0(X)\)-algebras.  Then an asymptotic morphism from~\(A\) to~\(B\) is asymptotically \(\Cont_0(X)\)-equivariant if and only if it is approximately \(X\)\nb-equivariant.
\end{proposition}

\begin{proof}
  Clearly, an asymptotically \(\Cont_0(X)\)-equivariant asymptotic morphism satisfies~\eqref{eq:def1_ah_X_equiv} since \(A(U)=\Cont_0(U)A\) and \(\Cont_0(U)\Cont_\bo(T,B)\subseteq \Cont_\bo(T,B(U))\).  Conversely, let~\(\varphi\) be approximately \(X\)\nb-equivariant.  Let \(B_\infty^X \defeq \Cont_0(X)\cdot B_\infty \subseteq B_\infty\), this is a \(\Cont_0(X)\)-algebra.  We are going to show that \(\dot\varphi\bigl(\Cont_0(U)A\bigr)\) is contained in \(\Cont_0(U)\cdot B_\infty^X = \Cont_0(U)\cdot B_\infty\) for all \(U\in\Open(X)\).  This is equivalent to the \(\Cont_0(X)\)-linearity of \(\dot\varphi\colon A\to B_\infty^X\) by \cite{Meyer-Nest:Bootstrap}*{Proposition 2.11}.

  For any \(f\in\Cont_0(U)\) and any \(\varepsilon>0\), there are a relatively compact open subset \(U_\varepsilon\subseteq \overline{U}_\varepsilon\subseteq U\) and \(f_\varepsilon\in\Cont_0(U_\varepsilon)\) with \(\norm{f-f_\varepsilon}<\varepsilon\).  Since~\(A\) is a \(\Cont_0(X)\)-\(\Cst\)-algebra, the same approximation applies to all \(a\in A(U) = \Cont_0(U)\cdot A\).  Therefore, it suffices to prove \(\dot\varphi\bigl(A(U')\bigr) \subseteq \Cont_0(U)\cdot B_\infty\) for all relatively compact open subsets~\(U'\) of~\(U\) with \(\overline{U'}\subseteq U\) .

  Since there is a function~\(w\) in \(\Cont_0(U)\) with \(w(x)=1\) for all \(x\in U'\), we have
  \[
  \Cont_\bo\bigl(T,B(U')\bigr)
  \subseteq w\cdot \Cont_\bo(T,B)
  \subseteq \Cont_0(U)\cdot \Cont_\bo(T,B)
  \]
  for all \(n\in\N\).  Since~\(\varphi\) maps \(A(U')\) into \(\Cont_\bo\bigl(T,B(U')\bigr) + \Cont_0(T,B)\) by~\eqref{eq:def1_ah_X_equiv}, \(\dot\varphi\) maps \(A(U')\) into \(\Cont_0(U)\cdot B_\infty\) for all \(n\in\N\).
\end{proof}

\subsection{Homotopy of asymptotic morphisms}
\label{sec:homotopy_asy}

\begin{definition}
  \label{def:asm_X}
  A \emph{homotopy} of asymptotic morphisms from~\(A\) to~\(B\) is an asymptotic morphism from~\(A\) to \(\Cont([0,1],B)\).  Let \(\asm{A,B}\) denote the set of homotopy classes of approximately \(X\)\nb-equivariant asymptotic morphisms from~\(A\) to~\(B\).
\end{definition}

Equivalent asymptotic morphisms are homotopic.

We do not know whether there is a natural topology on~\(\Asm{A,B}\) such that \(\asm{A,B} = \pi_0(\Asm{A,B})\).  It is easy to avoid this question by using quasi-topological spaces in the sense of Edwin H. Spanier (see~\cite{Spanier:Quasi-topologies}).

\begin{definition}
  \label{def:quasi-topological}
  A \emph{quasi-topological space} is a set~\(W\) together with distinguished sets of maps \(\Cont(Y,W)\) from~\(Y\) to~\(W\) for each compact Hausdorff space~\(Y\), called \emph{quasi-continuous maps} \(Y\to W\).  These quasi-continuous maps are required to satisfy the following conditions:
  \begin{itemize}
  \item constant maps are quasi-continuous;
  \item a function defined on a disjoint union \(Y_1\sqcup Y_2\) is quasi-continuous if and only if its restrictions to \(Y_1\) and~\(Y_2\) are quasi-continuous;

  \item if \(f\colon Y_1\to Y_2\) is a quasi-continuous map and \(h\colon Y_2\to W\) is quasi-continuous, so is \(h\circ f\); and, conversely,

  \item if~\(f\) is surjective and continuous (so that~\(f\) is an open surjection), then~\(h\) is quasi-continuous provided \(h\circ f\) is quasi-continuous.
  \end{itemize}
\end{definition}

Since~\(W\) is the set of quasi-continuous functions from the one-point space to~\(W\), we may also view a quasi-topological space as a contravariant functor from the category of compact Hausdorff spaces to the category of sets with some additional properties.

We define a quasi-topology on \(\Asm{A,B}\) by letting
\[
\Cont(Y,\Asm{A,B}) \defeq \Asm{A,\Cont(Y,B)}
\]
for each compact Hausdorff space~\(Y\).

Furthermore, \(\Asm{A,B}\) has a canonical base point, the zero map.  Thus \(\Asm{A,B}\) becomes a pointed quasi-topological space.

Homotopy groups for pointed quasi-topological spaces may be defined as for ordinary topological spaces, using quasi-continuous maps instead of continuous maps.  By definition, \(\asm{A,B} = \pi_0(\Asm{A,B})\).

\subsection{\texorpdfstring{$\E$}{E}-theory: Definition and universal property}
\label{sec:E_def}

The original approach of Alain Connes and Nigel Higson in~\cite{Connes-Higson:Deformations} only works well for separable \(\Cst\)\nb-algebras.  The same restriction applies to our equivariant generalisation.  Hence we (tacitly) assume all \(\Cst\)\nb-algebras to be separable from now on.  For similar reasons, we assume the underlying space~\(X\) to be second countable, that is, its topology must have a countable basis.

\begin{definition}
  \label{def:E_X}
  Let~\(X\) be a second countable topological space and let \(A\) and~\(B\) be separable \(\Cst\)\nb-algebras over~\(X\).  Following~\cite{Connes-Higson:Deformations}, we define
  \[
  \E_0(X;A,B) \defeq \asm{\Cont_0(\R,A)\otimes\Comp,\Cont_0(\R,B)\otimes\Comp}.
  \]
  The orthogonal direct sum turns \(\E_0(X;A,B)\) into an Abelian group.  This also holds for $\E_1(X;A,B) \defeq \E_0(X;\Cont_0(\R,A),B)$.
\end{definition}

\begin{proposition}
  \label{prop:composition_ah_X}
  The composition of asymptotic morphisms induces a product
  \[
  \asm{A,B} \times \asm{B,C} \to \asm{A,C}.
  \]
\end{proposition}

The proof is similar to the non-equivariant case outlined in~\cite{Connes:NCG}.  In addition to the arguments from \cite{Connes:NCG}*{Appendix~B of Chapter~II}, we need the following lemma to take care of approximate \(X\)\nb-equivariance.

Recall that an asymptotic morphism~\(\varphi\) is called \emph{uniformly continuous} if the map \(\varphi\colon A \to \Cont_\bo(T,B)\) is continuous.  By the Bartle--Graves Theorem, every asymptotic morphism is equivalent to a uniformly continuous one.

\begin{lemma}
  \label{lem:composition_X-equivariant}
  Let~\(X\) be a second countable topological space, let \(A\), \(B\) and~\(C\) be separable \(\Cst\)\nb-algebras, and  let \(\varphi\colon A \to \Cont_\bo(T,B)\) and \(\psi\colon B \to \Cont_\bo(T,C)\) be uniformly continuous, approximately \(X\)\nb-equivariant asymptotic morphisms.  Let~\(A_0\) be a \(\sigma\)\nb-compact dense \Star{}subalgebra of~\(A\).  There is an increasing, continuous map \(r_0\colon T \to T\) such that for any other increasing, continuous map \(r\colon T\to T\) with \(r(t)\ge r_0(t)\) for all \(t\in T\), there is an approximately \(X\)\nb-equivariant asymptotic morphism \(\theta\colon A \to \Cont_\bo(T,C)\) such that \(\lim_{t\to \infty} {}\norm{\theta_t(a)-\psi_{r(t)}\circ \varphi_t(a)}=0\) for all \(a\in A_0\).
\end{lemma}

\begin{proof}
  Let \((U_i)_{i=1}^\infty\) be a basis of open sets for the topology of~\(X\).  Choose a dense sequence \((a_{ij})_{j=1}^\infty\) in \(A(U_i)\) for each \(i\ge 1\).  We will find a map~\(r_0\) such that, for all \(r\ge r_0\),
  \begin{enumerate}[label=(\roman{*})]
  \item \((\psi_{r(t)}\varphi_t)\) is a bounded asymptotic morphism from~\(A_0\) to~\(C\), and
  \item \(\lim_{t\to \infty} {}\norm{\psi_{r(t)}\circ \varphi_t(a_{ij})}_{X\setminus U_i}=0\) for all \(i,j\).
  \end{enumerate}
  Then \(\psi_{r(t)}\circ\varphi_t\) defines a bounded \Star{}homomorphism \(A_0\to C_\infty\) by~(i).  It extends to a \Star{}homomorphism~\(\dot\theta\) on~\(A\).  Let \(\theta\colon A \to \Cont_\bo(T,C)\) be a lifting of~\(\dot\theta\).  Then~\(\theta\) is approximately \(X\)\nb-equivariant by~(ii).

  It remains to construct~\(r_0\).  By the usual non-equivariant case, there is a continuous map~\(r_{00}\) such that~(i) holds for all \(r\ge r_{00}\).  Since \(\varphi\bigl(A(U_i)\bigr)\subseteq \Cont_\bo\bigl(T,B(U_i)\bigr)+\Cont_0(T,B)\), there are \(f_{ij}\in \Cont_\bo\bigl(T,B(U_i)\bigr)\) and \(g_{ij}\in \Cont_0(T,B)\) such that \(\varphi(a_{ij})=f_{ij}+g_{ij}\) for all \(i,j\ge 1\).  Consider the following countable families of compact sets:
  \begin{align*}
    K_n&\defeq \bigcup_{i,j=1}^n f_{ij}[1,n+1] \cup
    g_{ij}[1,n+1]\subseteq B,\\
    L_{i,n} &\defeq \bigcup_{j=1}^{n} f_{ij}[1,n+1] \subseteq
    B(U_i).
  \end{align*}
  Since~\(\psi\) is a uniformly continuous asymptotic morphism, we can inductively construct an increasing sequence~\((s_n)_n\) such that for any \(s\ge s_n\)
  \begin{alignat}{2}
    \label{eq:psi_additive}
    \norm{\psi_s(x+y)-\psi_s(x)-\psi_s(y)}&<1/n,
    &\qquad&\text{for all \(x,y \in K_n\),}\\
    \label{eq:psi_bounded}
    \norm{\psi_s(x)} &<\norm{x} + 1/n,
    &\qquad&\text{for all \(x \in K_n\).}
  \end{alignat}
  Since~\(\psi\) is approximately \(X\)\nb-equivariant and \(L_{i,n}\subseteq B(U_i)\), for each~\(i\) there is an increasing sequence~\((r_{i,n})_n\) such that
  \begin{equation}
    \label{eq:psi_residual}
    \norm{\psi_s(x)}_{X\setminus U_i}< 1/n,
    \qquad\text{for all \(x \in L_{i,n}\) and all \(s\ge r_{i,n}\).}
  \end{equation}
  Choose an increasing continuous map \(r_0\colon T\to T\) with \(r_0(t)\ge r_{00}(t)\) and \(r_0(n)\ge \max\{s_n,r_{1,n},r_{2,n},\dotsc,r_{n,n}\}\) for all \(n\ge 1\).  We claim that any increasing, continuous function \(r\ge r_0\) satisfies~(ii).  This will finish the proof.

  Fix \(i,j\) and \(\varepsilon>0\).  Choose~\(n\) such that \(n\ge i\), \(n\ge j\) and \(1/n<\varepsilon/3\).  We shall show that for any \(t\ge n\),
  \[
  \norm{\psi_{r(t)}\circ \varphi_t(a_{ij})}_{X\setminus U_i} < \varepsilon + \norm{g_{ij}(t)}.
  \]
  This will conclude the proof since \(\lim_{t\to \infty} g_{ij}(t)=0\) by construction.  If \(t\ge n\), then there is an integer $m\geq n$ such that $m\le t < m+1$.  Therefore $f_{ij}(t)$ and~$g_{ij}(t)$ are in~$K_m$ and $r(t)\geq r(m)\geq s_m$.  Equation~\eqref{eq:psi_additive} yields
  \begin{equation}
    \label{eq:psi_additive_ep}
    \norm{\psi_{r(t)}(f_{ij}(t)+g_{ij}(t)) -
      \psi_{r(t)}(f_{ij}(t))-\psi_{r(t)}(g_{ij}(t))}
    <1/m<\varepsilon/3.
  \end{equation}
  Since \(i,j\le n \le m\) and $t<m+1$, we have \(f_{ij}(t)\in L_{i,m}\) and \(r(t)\ge r(m)\ge r_{i,m}\).  Inequality~\eqref{eq:psi_residual} yields
  \begin{equation}
    \label{eq:psi_residual_ep}
    \norm{\psi_{r(t)}(f_{ij}(t))}_{X\setminus U_i}<1/m <\varepsilon/3.
  \end{equation}
  Similarly, \eqref{eq:psi_bounded} yields
  \begin{equation}
    \label{eq:psi_bounded_ep}
    \norm{\psi_{r(t)}(g_{ij}(t))}\le \norm{g_{ij}(t)}+1/m<\norm{g_{ij}(t)}+\varepsilon/3.
  \end{equation}
  Putting together \eqref{eq:psi_additive_ep}, \eqref{eq:psi_residual_ep} and~\eqref{eq:psi_bounded_ep}, we get
  \begin{align*}
    \norm{\psi_{r(t)}\varphi_t(a_{ij})}_{X\setminus U_i}
    &\le \norm{\psi_{r(t)}(f_{ij}(t)+g_{ij}(t))-\psi_{r(t)}(f_{ij}(t))-\psi_{r(t)}(g_{ij}(t))}\\
    &\qquad+\norm{\psi_{r(t)}(f_{ij}(t))}_{X\setminus U_i}+\norm{\psi_{r(t)}(g_{ij}(t))}\\
    &<\varepsilon+\norm{g_{ij}(t)}.\qedhere
  \end{align*}
\end{proof}

For any extension of separable \(\Cst\)\nb-algebras \(I \into A \stackrel{p}\prto B\), there is a canonical asymptotic morphism from \(\Cont_0\bigl((0,1),B\bigr)\) to~\(I\).  If~$A$ is a \(\Cst\)\nb-algebra over~$X$, then $I$ and~$B$ become \(\Cst\)\nb-algebras over~$X$ in a unique natural way, such that the given extension is an extension of \(\Cst\)\nb-algebras over~\(X\).  Specifically, $I(U)=I\cap A(U)$ and $B(U)=p(A(U))$ for all~$U$ open in~$X$.

\begin{proposition}
  \label{proposition:extensions_X}
  Let \(I \into A \prto B\) be an extension of \(\Cst\)\nb-algebras over~\(X\).  Then the associated asymptotic morphism from \(\Cont_0\bigl((0,1),B\bigr)\) to~\(I\) is approximately \(X\)\nb-equivariant.
\end{proposition}

\begin{proof}
  Having an extension of \(\Cst\)\nb-algebras over~\(X\) means that we have \(\Cst\)\nb-algebra extensions
  \[
  I(U) \into A(U) \prto B(U)
  \]
  for all open subsets~\(U\) of~\(X\).  Since the map \(B(U)\to B\) is injective, this implies \(I(U) = I\cap A(U)= I\cdot A(U)\).

  We fix a positive and contractive continuous approximate unit \((u_t)_{t\in T}\) of~\(I\) which is quasi-central in~\(A\).  The canonical asymptotic morphism
  \[
  \gamma\colon SB \defeq \Cont_0\bigl((0,1),B\bigr) \to \Cont_\bo(T,I)
  \]
  is defined in two steps.  First, we define a homomorphism
  \[
  \gamma'\colon SA\to \Cont_\bo(T,I) \bigm/ \Cont_0(T,I),\qquad
  \gamma'_t (f\otimes a) \defeq f(u_t)\cdot a.
  \]
  Secondly, since the restriction of~\(\gamma'\) to~\(SI\) is equivalent to the null asymptotic morphism, \(\gamma'\)~induces an asymptotic morphism from~\(SB\) to~\(I\).  Clearly, $\gamma'$~is approximately \(X\)\nb-equivariant because \(I\cdot A(U)\subseteq I(U)\).  This is inherited by~\(\gamma\) because \(\dot\gamma\circ p=\dot\gamma'\), where $p\colon A\to B$ is the quotient map.
\end{proof}

Let $I\into B \stackrel{p}\prto C$ be an extension of \(\Cst\)\nb-algebras over~\(X\).  Let~$A$ be a \(\Cst\)\nb-algebra over~\(X\) and let $\varphi\colon A \to C$ be an \(X\)\nb-equivariant \Star{}homomorphism.  Let~$E$ be the \(\Cst\)\nb-algebra defined by the pullback diagram
\[
\xymatrix{
  0\ar[r]&
  I\ar[r]\ar@{=}[d] &
  E\ar[r]\ar[d]&
  A\ar[r]\ar[d]^{\varphi}&
  0\\
  0\ar[r] &
  I\ar[r]&
  B\ar[r]^{p}&
  C\ar[r]&
  0,
}
\]
that is, $E=\{(a,b)\in A\oplus B : \varphi(a)=p(b)\}$.  For $U\in \Open(X)$, set $E(U) \defeq E\cap \bigl(A(U)\oplus B(U)\bigr)$.

\begin{lemma}
  \label{lemma:pullbacKs_are_OK}
  $E$~is a \(\Cst\)\nb-algebra over~\(X\) and $I\into E \prto A$ is an extension of \(\Cst\)\nb-algebras over~\(X\).  The same conclusions hold if $B$ and~$C$ are only quasi \(\Cst\)\nb-algebras over~\(X\).
\end{lemma}

\begin{proof}
  Recall that for a quasi \(\Cst\)\nb-algebra~$B$ over~\(X\), the map $U\mapsto B(U)$ preserves only finite suprema in general.  The map $U\mapsto E(U)$ is obviously order-preserving.  The conditions $E(\emptyset)=0$, $E(X)=E$ and $E(U_1\cap U_2)=E(U_1)\cap E(U_2)$ are easily verified.  Let us show that $E(U_1\cup U_2) \subset E(U_1)+E(U_2)$, the reverse inclusion being obvious.  Let $(a,b)\in E(U_1\cup U_2)$.  Then $a\in A(U_1\cup U_2)=A(U_1)+A(U_2)$ and hence there are $a_i\in A(U_i)$, $i=1,2$ such that $a=a_1+a_2$.  Since~$\varphi$ is $X$\nb-equivariant, $\varphi(a_i)\in C(U_i)$ and hence there are $b_i\in B(U_i)$ such that $p(b_i)=\varphi(a_i)$, $i=1,2$.  It follows that $b_1+b_2-b\in B(U_1\cup U_2)$ and $p(b_1+b_2-b)=\varphi(a_1)+\varphi(a_2)-\varphi(a)=0$.  Therefore, $b_1+b_2-b \in I\cap B(U_1\cup U_2)=I(U_1\cup U_2)=I(U_1)+I(U_2)$.  This shows that there are $x_i\in I(U_i)$, $i=1,2$, such that $b_1+b_2-b=x_1+x_2$.  It follows that $(a_i,b_i-x_i)\in E(U_i)$ and $(a,b)=(a_1,b_1-x_1)+(a_2,b_2-x_2)$.

  It remains to show that \(E\bigl( \bigcup U_n\bigr)\) is the closure of \(\bigcup E(U_n)\) for any increasing sequence~\((U_n)\) in \(\Open(X)\).  The sequence of \(\Cst\)\nb-algebras
  \[
  I(U) \into E(U) \prto A(U)
  \]
  is exact for each open set~\(U\).  Since \(A\) and~\(I\) are \(\Cst\)\nb-algebras over~\(X\),
  \begin{gather*}
    A(U) = \cl{\bigcup A(U_n)} = \varinjlim A(U_n),\\
    I(U) = \cl{\bigcup I(U_n)} = \varinjlim I(U_n).
  \end{gather*}
  Since the \(\Cst\)\nb-algebra inductive limit functor is exact, we get another extension of \(\Cst\)\nb-algebras
  \[
  I(U)\into \cl{\bigcup E(U_n)} \prto A(U)
  \]
  because \(\varinjlim E(U_n) = \cl{\bigcup E(U_n)}\).  This implies that \(E(U)\) is the supremum of \(\bigl\{E(U_n)\}\), so that~\(E\) is a \(\Cst\)\nb-algebra over~\(X\).
\end{proof}

\begin{theorem}
  \label{the:equivariant_E_universal}
  The equivariant \(\E\)\nb-theory defined above carries a composition product and hence yields a category~\(\Ecat(X)\).  The canonical functor from the category \(\Cstarsep(X)\) of separable \(\Cst\)\nb-algebras over~\(X\) to~\(\Ecat(X)\) is the universal half-exact, \(\Cst\)\nb-stable homotopy functor.
\end{theorem}

\begin{proof}
  The composition product is described in Proposition~\ref{prop:composition_ah_X}.  The same argument as in the non-equivariant case shows that it is associative.  The functor \(\Cstarsep(X)\to\Ecat(X)\) is a \(\Cst\)\nb-stable homotopy functor by definition.  Next we check its exactness.

  Let \(I \into E \overset{p}\prto Q\) be an extension of \(\Cst\)\nb-algebras over \(X\).  The cone
  \begin{alignat*}{2}
    C_p&\defeq \{(f,a)\in \Cont_0((0,1],Q)\oplus E: f(1)=p(a)\},\\
    C_p(U)&\defeq\bigl(\Cont_0((0,1],Q(U))\oplus E(U) \bigr)\cap C_p&\qquad&\text{for \(U\in\Open(X)\)},
  \end{alignat*}
  is a \(\Cst\)\nb-algebra over~\(X\) by Lemma~\ref{lemma:pullbacKs_are_OK}.  The asymptotic morphism \(\gamma_t\colon SC_p\to SI\) induced by the extension \(SI \into CE \prto C_p\) is approximately \(X\)\nb-equivariant.  There is a natural \(X\)\nb-equivariant inclusion \(i\colon I\to C_p\), \(i(a)=(0,a)\).  The proof of \cite{Dadarlat:Asymptotic}*{Theorem~13} with no essential change yields that~\(\gamma\) is a homotopy inverse of~\(Si\), that is, \(\asm{\gamma\circ Si} = \asm{\id_{SI}}\) and \(\asm{Si\circ \gamma} = \asm{\id_{SC_p}}\).  As in the non-equivariant case, this excision result and Proposition~\ref{prop:composition_ah_X} show that \(\E_0(X;A,B) \defeq \asm{SA\otimes \Comp, SB\otimes \Comp}\) is a periodic exact functor in both variables \(A\) and~\(B\), that is, if \(I\into E\prto Q\) is an extension in \(\Cstarsep(X)\) and~\(B\) is a separable \(\Cst\)\nb-algebra over~\(X\), then there are six-term exact sequences
  \[
  \xymatrix{
    \E_0(X;Q,B)\ar[r]&
    \E_0(X;E,B)\ar[r]&
    \E_0(X;I,B)\ar[d]^\partial\\
    \E_1(X;I,B)\ar[u]^\partial&
    \E_1(X;E,B)\ar[l]&
    \E_1(X;Q,B)\ar[l]
  }
  \]
  and
  \[
  \xymatrix{
    \E_0(X;B,I)\ar[r]&
    \E_0(X;B,E)\ar[r]&
    \E_0(X;B,Q)\ar[d]^\partial\\
    \E_1(X;B,Q)\ar[u]^\partial&
    \E_1(X;B,E)\ar[l]&
    \E_1(X;B,I).\ar[l]
  }
  \]
  The horizontal maps in both exact sequences are induced by the given maps \(I\to E\to Q\), and the vertical maps are, up to signs, products with the class of the approximately \(X\)\nb-equivariant asymptotic morphism associated to the extension as in Proposition \ref{proposition:extensions_X}.

  It remains to verify universality.  Again this is similar to the proof of the non-equivariant case in \cite{Blackadar:K-theory}*{Theorem~25.6.1}, using Lemma~\ref{lemma:universality_E_X} below as a substitute for \cite{Blackadar:K-theory}*{Proposition~25.6.2}.
\end{proof}

\begin{lemma}
  \label{lemma:universality_E_X}
  Any element of \(\E_0(X;A,B)\) may be written as \([\rho]\circ [\pi]^{-1}\) for \(X\)\nb-equivariant \Star{}homomorphisms \(\rho\) and~\(\pi\).
\end{lemma}

\begin{proof}
  Let \(\varphi\colon A\to \Cont_\bo(T,B)\) be an approximately \(X\)\nb-equivariant asymptotic morphism.  We shall use Lemma~\ref{lemma:pullbacKs_are_OK} to show that the \(\Cst\)\nb-algebra
  \[
  E\defeq \{(a,b)\in A \oplus \Cont_\bo(T,B):
  \varphi(a)-b \in \Cont_0(T,B)\},
  \]
  becomes a \(\Cst\)\nb-algebra over~\(X\) by
  \[
  E(U) \defeq E\cap \bigl(A(U)\oplus \Cont_\bo(T,B(U))\bigr).
  \]
  As a consequence of the Bartle--Graves Theorem, for any two closed two-sided ideals $J_1$ and~$J_2$ in a $\Cst$\nb-algebra~$D$, $\Cont_b\bigl(T,J_1+J_2\bigr)=\Cont_b\bigl(T,J_1\bigr)+\Cont_b\bigl(T,J_2\bigr)$.  From this we see that
  \[
  \Cont_0\bigl(T,B\bigr)\into \Cont_b\bigl(T,B\bigr)\prto \Cont_b\bigl(T,B\bigr)/\Cont_0\bigl(T,B\bigr)=B_\infty
  \]
  is an extension of quasi \(\Cst\)\nb-algebras over~$X$.  By Lemma~\ref{lemma:pullbacKs_are_OK}, its pullback under the $X$\nb-equivariant \Star{}homomorphism $\dot{\varphi}\colon A\to B_\infty$ is an extension of \(\Cst\)\nb-algebras over~\(X\):
  \[
  \Cont_0(T,B) \into E\overset{\pi}\prto A
  \]
  with \(\pi(a,b)\defeq \varphi(a)\).  The map~\(\pi\) becomes an isomorphism in \(\Ecat(X)\) because \(\Cont_0(T,B)\) is contractible over~\(X\).  Let \(\rho'\colon E \to \Cont_\bo(T,B)\) be the \Star{}homomorphism \(\rho'(a,b)=b\).  When regarded as an asymptotic morphism from~\(E\) to~\(B\), \(\rho'\) is homotopic to the constant asymptotic morphism \(\rho (a,b)=b(0)\).  We have \([\varphi]\circ [\pi]=[\rho']\) because \(\varphi\bigl(\pi(a,b)\bigr)-\rho'(a,b)\in \Cont_0(T,B)\) for all \((a,b)\in E\).  Hence \([\varphi]=[\rho]\circ [\pi]^{-1}\).
\end{proof}

\subsection{Further properties of \texorpdfstring{$\E$}{E}-theory}
\label{sec:further_properties}

Like the category \(\KKcat(X)\), the category \(\Ecat(X)\) carries the additional structure of a triangulated category (see \cites{Meyer-Nest:BC, Neeman:Triangulated}).  As in \(\KK\)-theory, the translation automorphism is the suspension functor \(A\mapsto SA\defeq \Cont_0\bigl((0,1),A\bigr)\), and a triangle is exact if it is isomorphic to the mapping cone triangle of some \(X\)\nb-equivariant \Star{}homomorphism.

\begin{theorem}
  \label{the:E-triangulated}
  The category \(\Ecat(X)\) is triangulated.
\end{theorem}

\begin{proof}
  The argument is essentially the same as in the appendix of~\cite{Meyer-Nest:BC}.  The only axiom that requires a different treatment is the one that requires each \(\varphi\in\E_0(X;A,B)\) to embed in an exact triangle.  Here we use the factorisation \(\varphi= [\rho]\circ[\pi]^{-1}\) of Lemma~\ref{lemma:universality_E_X} with \(X\)\nb-equivariant \Star{}homomorphisms \(\rho\colon E\to B\) and \(\pi\colon E\to A\).  Since~\([\pi]\) is invertible in \(\E\)\nb-theory, the mapping cone triangle
  \[
  SB \to C_\rho \to E \xrightarrow{\rho} B
  \]
  is isomorphic to an exact triangle \(SB \to C_\rho \to A\xrightarrow{[\varphi]} B\).
\end{proof}

The proof that \(\E\)\nb-theory is exact shows that any extension \(I\into E\prto Q\) of \(\Cst\)\nb-algebras over~\(X\) gives rise to an exact triangle \(SQ\to I\to E\to Q\), where the map \(SQ\to I\) is the Connes--Higson construction (see Proposition~\ref{proposition:extensions_X}) and the maps \(I\to E\to Q\) are the given ones.  Such triangles are called \emph{extension triangles}.  This works for all extensions, so that we need no admissibility assumption as in \(\KKcat(X)\).

Since there is no admissibility hypothesis, several constructions in Kasparov theory simplify in \(\E\)\nb-theory.  For instance, the colimit \(\varinjlim {}(A_n,\varphi_n)\) of any inductive system \(\varphi_n\colon A_n\to A_{n+1}\), \(n\in\N\), in \(\Cstarsep(X)\) is also a homotopy colimit in \(\Ecat(X)\), by the argument in \cite{Meyer-Nest:BC}*{Section 2.4}.

\begin{proposition}
 \label{prop:Milnor_regular}
 If~$A$ is the inductive limit of an inductive system \((A_n,\varphi_n)\) in \(\Cstarsep(X)\), then there is a natural short exact sequence
 \[
 0\to \varprojlim\nolimits^1 {}\E_1(X;A_n,B) \to \E(X;A,B)\to \varprojlim {}\E(X;A_n,B)\to 0.
 \]
\end{proposition}

\begin{proof}
  The functor $A\mapsto \E(X;A,B)$ is seen to be countably additive as in the proof of \cite{Guentner-Higson-Trout:Equivariant_E}*{Proposition~7.1}.  Then we follow the standard argument based on mapping telescopes in \cite{Blackadar:K-theory}*{Section 21.3.2}.
\end{proof}

For locally compact Hausdorff spaces, we may compare our definition of equivariant \(\E\)\nb-theory with previous ones in \cites{Popescu:Equivariant_E, Park-Trout:Representable_E}.  Since we use the original Connes--Higson model of \(\E\)\nb-theory instead of iterated asymptotic algebras, this does not yet follow directly from Proposition~\ref{pro:compare_Popescu} and~\cite{Popescu:Equivariant_E}.

\begin{proposition}
  \label{pro:E_Hausdorff_compare}
  Let~\(X\) be Hausdorff and locally compact and let \(A\) and~\(B\) be \(\Cst\)\nb-algebras over~\(X\).  Then \(\E_*(X;A,B)\) is naturally isomorphic to \(\RE_*(X;A,B)\).
\end{proposition}

\begin{proof}
  Both theories satisfy the same universal property.  Alternatively, the statement follows from Proposition~\ref{pro:compare_Popescu} and~\cite{Park-Trout:Representable_E}.
\end{proof}

Recall that for a compact Hausdorff space~\(X\), there is a canonical isomorphism
\[
\KK_*(X;\Cont(X,A),B) \cong \KK_*(A,B)
\]
for any \(\Cst\)\nb-algebra~\(A\) and any \(\Cst\)\nb-algebra~\(B\) over~\(X\).  The same isomorphism holds in \(\E\)\nb-theory as well:

\begin{lemma}
  \label{lem:E_compact_Hausdorff}
  Let~\(X\) be a compact Hausdorff space.  Then
  \[
  \E_*(X;\Cont(X,A),B) \cong \E_*(A,B)
  \]
  for any \(\Cst\)\nb-algebra~\(A\) and any \(\Cst\)\nb-algebra~\(B\) over~\(X\).
\end{lemma}

\begin{proof}
  We may view \(\Cont(X,A)\) as a \(\Cst\)\nb-algebra over~\(X\) using the obvious map \(\Prim \Cont(X,A)\to X\), so that \(\Cont(X,A)(U) \defeq \Cont_0(U,A)\) for \(U\in\Open(X)\).  We have to show that the functor
  \[
  \Ecat\to\Ecat(X),\qquad A\mapsto \Cont(X,A),
  \]
  is left adjoint to the functor
  \[
  \Ecat(X)\to\Ecat,\qquad B\mapsto B(X).
  \]
  First of all, both maps on objects clearly induce functors on \(\E\)\nb-theory categories because of the universal properties.  For the adjointness, we have to furnish the unit and counit of adjunction and verify the two zigzag equations (see~\cite{MacLane:Categories}).  The unit is the \(X\)\nb-equivariant \Star{}homomorphism \(\Cont(X,B) = \Cont(X)\otimes B(X) \to B\) that comes from viewing a \(\Cst\)\nb-algebra~\(B\) over~\(X\) as a \(\Cont(X)\)-\(\Cst\)\nb-algebra.  The counit is the embedding \(A\to \Cont(X,A)(X) = \Cont(X,A)\), \(a\mapsto 1\otimes a\), as constant functions.  The zigzag equations are trivial to verify and hold already on the level of \Star{}homomorphisms.
\end{proof}

\begin{proposition}
  \label{pro:restrict_to_locally_closed}
  Let \(Y\subseteq X\) be a locally closed subset.  Then there exists a natural restriction functor \(\E_*(X;A,B)\to \E_*\bigl(Y;r_X^Y(A),r_X^Y(B)\bigr)\) for \(\Cst\)\nb-algebras \(A\) and~\(B\) over~\(X\).
\end{proposition}

\begin{proof}
  The restriction functor \(\Cstarsep(X)\to\Cstarsep(Y)\) is defined in~\cite{Meyer-Nest:Bootstrap} by \(r_X^YA(Z)\defeq A(Y\cap Z)\) for \(Z\in\Open(Y)\).  It evidently maps extensions again to extensions and commutes with stabilisation.  Hence it induces a functor on \(\E\)\nb-theory by the universal property.
\end{proof}

\section{Approximation by finite spaces}
\label{sec:approx_finite}

Let \(\cov = (U_n)_{n\in\N}\) be a countable basis for the topology of~\(X\).  For each \(n\in\N\), let~\(\tau_n\) be the topology generated by the open subsets \(U_1,\dotsc,U_n\).  That is, the subsets~\(U_j\) are a subbasis for~\(\tau_n\), so that the intersections
\[
U_F \defeq \bigcap_{i\in F} U_i
\]
for \(F\subseteq \{1,\dotsc,n\}\) are a basis for~\(\tau_n\).

Since the topology~\(\tau_n\) is finite, it is pulled back from a finite \(\textup T_0\)\nb-space~\(X_n\); namely, we equip~\(X\) with the equivalence relation
\[
x\sim_n y \iff \{1\le j\le n : x\in U_j\} = \{1\le j\le n : y\in U_j\}
\]
for \(x,y\in X\).  We may view~\(\tau_n\) as a topology on the finite set \(X/{\sim_n}\).  A point in \(X/{\sim_n}\) is parametrised by the set \(\{1\le j\le n : x\in U_j\}\).

\begin{remark}
  \label{rem:min_X_n}
  The minimal open neighbourhood in~\(X_n\) that contains the point corresponding to \(F\subseteq\{1,\dotsc,n\}\) is the image in~\(X_n\) of \(U_F\defeq \bigcap_{i\in F} U_i\).
\end{remark}

In the following, we view \(\Cst\)\nb-algebras over~\(X\) as \(\Cst\)\nb-algebras over \((X,\tau_n)\) or, equivalently, over \(X_n\defeq (X/{\sim_n},\tau_n)\) by forgetting most of the distinguished ideals.

\begin{theorem}
  \label{the:approximate_E}
  Let \(A\) and~\(B\) be \(\Cst\)\nb-algebras over~\(X\), viewed as \(\Cst\)\nb-algebras over \(X_n\defeq (X/{\sim_n},\tau_n)\) for \(n\in\N\).  Then there is a natural extension of \(\Z/2\)-graded Abelian groups
  \[
  \varprojlim\nolimits^1 \E_{*+1}(X_n;A,B)
  \into \E_*(X;A,B) \prto
  \varprojlim \E_*(X_n;A,B).
  \]
\end{theorem}

\begin{proof}
  Recall the description of \(\asm{A,B}\) as the zeroth homotopy group of a quasi-topological space \(\Asm{A,B}\) in Section~\ref{sec:homotopy_asy}.  This also applies to \(\E\)\nb-theory: we have \(\E_0(X;A,B) \cong \pi_0(\Gamma_X)\) with
  \[
  \Gamma_X \defeq \Asm{\Cont_0(\R,A)\otimes\Comp, \Cont_0(\R,B)\otimes\Comp}.
  \]
  The same definitions for~\(X_n\) yield quasi-topological spaces \(\Gamma_n \defeq \Gamma_{X_n}\) for \(n\in\N\) with \(\E_0(X_n;A,B) \cong \pi_0( \Gamma_n)\).  The quasi-topological spaces \(\Gamma_n\) form a projective system because approximate \(X_{n+1}\)\nb-equivariance implies approximate \(X_n\)\nb-equivariance.

  We claim that
  \[
  \Gamma_X = \bigcap_{n\in\N} \Gamma_n,\qquad
  \Cont(Y,\Gamma_X) = \bigcap_{n\in\N} \Cont(Y,\Gamma_n)
  \]
  for each compact Hausdorff space~\(Y\), where \(\Cont(Y,\Gamma_n)\) denotes the space of quasi-continuous maps \(Y\to\Gamma_n\).

  The inclusion \(\Cont(Y,\Gamma_X)\subseteq \bigcap \Cont(Y,\Gamma_n)\) is evident.  The intersection of \(\Cont(Y,\Gamma_n)\) consists of those asymptotic morphisms that satisfy~\eqref{eq:def1_ah_X_equiv} for all \(U\in\cov\).  Since the set of open subsets for which~\eqref{eq:def1_ah_X_equiv} holds is closed under arbitrary unions and~\(\cov\) is a basis for the topology of~\(X\), this implies~\eqref{eq:def1_ah_X_equiv} for all open subsets of~\(X\), proving the claim.

  The claim above shows that \(\Gamma_X\) is the inverse limit of the projective system \(\Gamma_n\).  The homotopy groups of inverse limits of ordinary topological spaces are computed by an exact sequence of the desired form if the maps \(\Gamma_{n+1}\to \Gamma_n\) have the \emph{homotopy covering property}, see~\cite{Vogt:Dual_Milnor}.  It is easy to see that this carries over to quasi-topological spaces; but in our case the maps \(\Gamma_{n+1} \to \Gamma_n\) are injective and therefore cannot have the homotopy covering property.  Nevertheless, we can get the desired result by following part of the argument in~\cite{Vogt:Dual_Milnor}.

  First we observe that \cite{Vogt:Dual_Milnor}*{Theorem~C}, which computes the homotopy groups of homotopy equalisers remains true for quasi-topological spaces.  Let \(f,g\colon A\rightrightarrows B\) be two base point preserving quasi-continuous maps between pointed quasi-topological spaces.  The \emph{homotopy equaliser} of \(f,g\) is the quasi-topological space \(D(f,g)\) defined so that, for all~\(Y\) compact Hausdorff,
  \begin{multline*}
    \Cont\bigl(Y,D(f,g)\bigr) =
    \{(a,b) \in \Cont(Y,A) \times \Cont(Y\times I,B)\mid\\
    f\circ a = b(\blank,0),\quad g\circ a = b(\blank,1)\}.
  \end{multline*}
  Let~\(Y\) be a compact Hausdorff space.  Then there is an exact sequence of pointed sets
  \begin{equation}
    \label{eq:htpy_equaliser}
    * \to T \to [Y,D(f,g)] \to K \to *
  \end{equation}
  where \([Y,X]\) denotes homotopy classes of quasi-continuous maps \(Y\to X\), \(K \defeq \{a \in [Y,A] \mid f_*(a) = g_*(a)\}\), and~\(T\) is the orbit space for a certain canonical action of \([Y\times\Sphere^1,A]_*\) on \([Y\times\Sphere^1,B]_*\), where \([Y\times\Sphere^1,\blank]_*\) means that we restrict attention to quasi-continuous maps and homotopies that map \(Y\times\{1\}\subseteq Y\times\Sphere^1\) to the base point.

  Next we apply~\eqref{eq:htpy_equaliser} to the pair of maps
  \[
  \Id,f\colon \prod_{n=0}^\infty \Gamma_n \rightrightarrows \prod_{n=0}^\infty \Gamma_n,
  \]
  where~\(f\) is the shift map from the definition of the projective limit.  Letting \(\gamma_{n+1}^n\colon \Gamma_{n+1}\to\Gamma_n\) denote the maps in the projective system, we have \(f\bigl((x_n)_{n\in\N})\bigr) \defeq \bigl(\gamma_{n+1}^n(x_{n+1})\bigr)_{n\in\N}\).  The homotopy equaliser of \((\id,f)\) is quasi-homeomorphic to the quasi-topological space~\(\Gamma_\infty\) defined by
  \begin{multline*}
    \Cont(Y,\Gamma_\infty) \defeq
    \Bigl\{ (f_n)_{n=0}^\infty \in \prod_{n\in\N}
    \Cont([0,1]\times Y,\Gamma_n) \Bigm|\\
    \text{\(f_n(1) = \gamma_{n+1}^n\bigl(f_{n+1}(0)\bigr)\) for
      all \(n\in\N\)} \Bigr\}.
  \end{multline*}
  This is a familiar mapping telescope construction.  The quasi-topological version of \cite{Vogt:Dual_Milnor}*{Theorem~C} shows that the homotopy groups of~\(\Gamma_\infty\) are computed by an exact sequence of exactly the desired form.

  To finish the proof of the theorem, it remains to show that the homotopy limit~\(\Gamma_\infty\) and the limit~\(\Gamma_X\) of the projective system~\((\Gamma_n)\) have isomorphic~\(\pi_0\).  Lacking the homotopy covering property used in~\cite{Vogt:Dual_Milnor}, we do this by hand.

  Let us describe the homotopy limit~\(\Gamma_\infty\) more concretely.  The maps $\gamma_{n+1}^n\colon \Gamma_{n+1}\to\Gamma_n$ are just the inclusion maps.  It is convenient to identify \(\Cont(Y,\Gamma_\infty)\) with
  \begin{multline*}
    \Cont(Y,\Gamma_\infty) =
    \Bigl\{ (f_n)_{n=0}^\infty \in \prod_{n\in\N}
    \Cont([n,n+1]\times Y,\Gamma_n)
    \Bigm|\\
    \text{\(f_n(n+1) = f_{n+1}(n+1)\) for all \(n\in\N\)} \Bigr\}.
  \end{multline*}
  We view each~$f_n$ as an approximately $X_n$\nb-equivariant asymptotic morphism from~$A'$ to $\Cont([n,n+1]\times Y,B')$, where \(A'\defeq \Cont_0(\R,A)\otimes\Comp\) and \(B'\defeq \Cont_0(\R,B)\otimes\Comp\).  We may piece together these asymptotic morphisms to a single family of maps \(\varphi_{s,t}\colon A'\to \Cont(Y,B')\), $s,t\in T$, where \(\varphi_{s,t}|_{s\in[n,n+1]}\) is~\(f_n\).  That is, \(\varphi_{s,t}\) is an asymptotic morphism for fixed~\(s\), uniformly for \(s\in[n,n+1]\) for all~\(n\), and hence uniformly for~\(s\) in compact subsets of~\(T\); furthermore, this asymptotic morphism is (uniformly) approximately \(X_n\)\nb-equivariant for \(s\in[n,n+1]\) and hence for~\(s\) in compact subsets of \([n,\infty)\).

  We map \(\Gamma_X\) to~\(\Gamma_\infty\) by taking a constant family of asymptotic morphisms.  It remains to show that this map \(\Gamma_X\to \Gamma_\infty\) induces an isomorphism on homotopy classes.

  Let \((\varphi_{s,t})\in\Gamma_\infty\) and let \(A_0\subseteq A'\) be a compactly generated dense subalgebra.  The same considerations as in the construction of the product of asymptotic homomorphisms show that there is an increasing continuous function \(h_0\colon T\to T\) such that \(\varphi_{t,h(t)}\colon A_0\to B'\) extends to an \(X\)\nb-equivariant asymptotic morphism for all continuous \(h\ge h_0\).  Here we use that an asymptotic morphism is \(X\)\nb-equivariant once it satisfies~\eqref{eq:def1_ah_X_equiv} for all \(U\in\cov\).  Furthermore, we may choose~\(h_0\) such that the convex homotopies~\(\varphi_{s,rh(t)+(1-r)t}\) from~\(\varphi_{s,t}\) to~\(\varphi_{s,h(t)}\) and~\(\varphi_{rt+(1-r)s,h(t)}\) from~\(\varphi_{s,h(t)}\) to~\(\varphi_{t,h(t)}\) are homotopies in~\(\Gamma_\infty\) for \(h\ge h_0\). We discuss this in detail below.  Thus~\((\varphi_{s,t})\) is homotopic to the constant family of asymptotic morphism~\((\varphi_{t,h(t)}) \) in~\(\Gamma_\infty\), so that the map \(\pi_0(\Gamma_X) \to \pi_0(\Gamma_\infty)\) is surjective.  A similar argument may be applied to homotopies in~\(\Gamma_\infty\) and shows that two elements of \(\Gamma_X(A,B)\) that become homotopic in~\(\Gamma_\infty\) are already homotopic in \(\Gamma_X\).

  Let us now show how to construct the function~$h_0$ for given \((\varphi_{s,t})\in\Gamma_\infty\).  The first homotopy from~\(\varphi_{s,t}\) to~\(\varphi_{s,h(t)}\) is a homotopy of asymptotic morphisms provided \(h(t)\ge t\), for obvious reasons.  Thus it only remains to study the second homotopy.  Let \(A_0=\{a_1,a_2,\dotsc\}\subseteq A'\) be a countable dense $^*$\nb-subalgebra.  Let $\{\lambda_1,\lambda_2,\dotsc\}$ be a sequence dense in~$\C$.  Let \((U_i)_{i=1}^\infty\) be a basis of open sets for the topology of~\(X\).  Choose a dense sequence \((a_{ij})_{j=1}^\infty\) in \(A'(U_i)\) for each \(i\ge 1\).

  For each integer $m\geq 1$ choose $\alpha_m>0$ such that for all $1\le i,j,k \le m$ and all $t\geq \alpha_m$,
  \begin{gather}
    \label{eq:add101}
    \sup_{s\in [0,m+1]} \norm{\varphi_{s,t}(a_i^*+\lambda_k a_j)-\varphi_{s,t}(a_i)^*-\lambda_k\,\varphi_{s,t}(a_j)}
    < 1/m,\\
    \label{eq:mult101}
    \sup_{s\in [0,m+1]} \norm{\varphi_{s,t}(a_ia_j)-\varphi_{s,t}(a_i)\varphi_{s,t}(a_j)}
    < 1/m.
  \end{gather}
  For each integer $n\geq 1$ we construct a sequence $(\tau_{m,n})_{m=1}^{\infty}$ such that
  \begin{equation}
    \label{eq:equiv101}
    \sup_{s\in [n,m+1]} \norm{\varphi_{s,t}(a_{ij})}_{X\setminus U_i}
    < 1/m,
  \end{equation}
  for all $1\le i\le n$, $1\le j \le m$ and all $t\geq \tau_{m,n}$.  Moreover, once the sequence $(\tau_{m,n})_{m=1}^\infty$ is constructed, we construct the next sequence $(\tau_{m,n+1})_{m=1}^\infty$ such that $\tau_{m,n+1}\geq\tau_{m,n}$ for all $m\geq 1$.  Let $h_0\colon T\to T$ be a continuous increasing function with $h_0(m) \geq \max\{\alpha_m, \tau_{m,m}\}$ and $\lim_{t\to\infty} h_0(t)=\infty$.

  Let $h\geq h_0$ be a continuous function.  The homotopy~\(\varphi_{rt+(1-r)s,h(t)}\) is defined by an element~$H$ in
  \begin{multline*}
    \Cont([0,1]\times Y,\Gamma_\infty) =
    \Bigl\{ (H_n)_{n=0}^\infty \in \prod_{n\in\N}
    \Cont([0,1]\times [n,n+1]\times Y,\Gamma_n)
    \Bigm|\\
    \text{\(H_n(n+1) = H_{n+1}(n+1)\) for all \(n\in\N\)} \Bigr\},
  \end{multline*}
  where for $r\in [0,1]$, $(H_n)_r\defeq (\varphi_{rt+(1-r)s,h(t)})_{s\in [n,n+1],t\in T}$.

  In order to verify that~$H$ is an element of $\Cont([0,1]\times Y,\Gamma_\infty)$, it is sufficient to show that for all $i,j,k\geq 1$
  \begin{multline}
    \label{eq:add102}
    \lim_{\strut t\to \infty} \sup_{s\in [n,n+1],\ r\in [0,1]} \norm{\varphi_{rt+(1-r)s,h(t)}(a_i^*+\lambda_k a_j)\\
    -\varphi_{rt+(1-r)s,h(t)}(a_i)^*-\lambda_k\,\varphi_{rt+(1-r)s,h(t)}(a_j)}=0,
  \end{multline}
  \begin{multline}
    \label{eq:mult102}
    \lim_{\strut t\to \infty} \sup_{s\in [n,n+1],\ r\in [0,1]}\norm{\varphi_{rt+(1-r)s,h(t)}(a_ia_j)\\
    -\varphi_{rt+(1-r)s,h(t)}(a_i)\ \varphi_{rt+(1-r)s,h(t)}(a_j)}=0,
  \end{multline}
  and that for all $1\le i \le n$, $j\geq 1$
  \begin{equation}
    \label{eq:equiv102}
    \lim_{\strut t\to \infty} \sup_{s\in [n,n+1],\ r\in [0,1]}\norm{\varphi_{rt+(1-r)s,h(t)}(a_{ij})}_{X\setminus U_i}
    =0.
  \end{equation}
  We deal first with \eqref{eq:add102} and~\eqref{eq:mult102}.  Let $i,j,k \geq 1$ and $\varepsilon>0$ be given.  We claim that for any $t\geq \max\{n,i,j,k, 1/\varepsilon\}+1$, the quantities whose limits are taken in \eqref{eq:add102} and~\eqref{eq:mult102} are smaller than~$\varepsilon$.  If~$m$ is the integer part of~$t$, then $\max\{n,i,j,k, 1/\varepsilon\}< m \le t<m+1$. Moreover, for any $s\in [n,n+1]$ and $r\in [0,1]$, $rt+(1-r)s\in [0,m+1]$ and $h(t)\geq h_0(t)\geq h_0(m)\geq \alpha_m$.  Since $1/m<\varepsilon$ our claim follows now from \eqref{eq:add101} and~\eqref{eq:mult101}.

  Let us now check~\eqref{eq:equiv102}. Let $1\le i \le n$, $j\geq 1$ and $\varepsilon>0$ be given and suppose that $t\geq \max\{n,j, 1/\varepsilon\}+1$.  Then there is an integer~$m$ such that $\max\{n,j, 1/\varepsilon\}< m \le t<m+1$.  Observe that for any $s\in [n,n+1]$ and $r\in [0,1]$, $rt+(1-r)s\in [n,m+1]$ and $h(t)\geq h_0(t)\geq h_0(m)\geq \tau_{m,m}\geq \tau_{m,n}$.  Since $1/m<\varepsilon$, it follows from~\eqref{eq:equiv101} that the quantity whose limit is taken in~\eqref{eq:equiv102} is smaller than~$\varepsilon$ whenever $t\geq \max\{n,j, 1/\varepsilon\}+1$.
\end{proof}

\begin{theorem}
  \label{the:E-iso_finite}
  Let~\(X\) be a second countable topological space.  An element in \(\E_*(X;A,B)\) is invertible if and only if its image in \(\E_*\bigl(A(U),B(U)\bigr)\) is invertible for all \(U\in\Open(X)\).
\end{theorem}

\begin{proof}
  The necessity of the condition is trivial.  Next we sketch
  why the condition is sufficient if~\(X\) is a finite space.
  The proof is similar to the proof of a similar statement in
  \(\KK\)\nb-theory in \cite{Meyer-Nest:Bootstrap}*{Proposition
    4.9}.  If~\(X\) is finite, any point \(x\in X\) is
  contained in a minimal open subset~\(U_x\).  For a
  \(\Cst\)\nb-algebra~\(A\), let~\(i_x A\) be~\(A\) viewed as a
  \(\Cst\)\nb-algebra over~\(X\) concentrated at \(x\in X\),
  that is, \(i_x(A)(U) = A\) for \(x\in U\) and \(i_x(A)(U) =
  0\) for \(x\notin U\).  An argument similar to the proof of
  \cite{Meyer-Nest:Bootstrap}*{Proposition 3.13} yields
  \[
  \E_*(X;i_x(A),B) \cong \E_*\bigl(A,B(U_x)\bigr)
  \]
  for \(x\in X\), a \(\Cst\)\nb-algebra~\(A\) and a
  \(\Cst\)\nb-algebra~\(B\) over~\(X\).  An argument similar to
  the proof of \cite{Meyer-Nest:Bootstrap}*{Proposition 4.7}
  shows that objects of the form~\(i_x(A)\)
  generate~\(\Ecat(X)\), that is, no proper triangulated
  subcategory of~\(\Ecat(X)\) contains \(i_x(A)\) for all~\(A\)
  (see also Proposition~\ref{pro:Bootstrap_generators_finite}
  below).  Hence a map in \(\E_*(X;A,B)\) is invertible if the
  induced map  \(\E_*(X;i_x(D),A)\to\E_*(X;i_x(D),B)\) is invertible for
  all \(x\in X\) and all~\(D\).  By the isomorphism above, this
  is equivalent to invertibility of the induced map
  \(\E_*\bigl(D,A(U_x)\bigr)\to\E_*\bigl(D,B(U_x)\bigr)\), which is equivalent to
  invertibility in \(\E_*\bigl(A(U_x),B(U_x)\bigr)\) for
  all~\(x\).  This finishes the argument for finite~\(X\).

  If~\(X\) is infinite, let~\(\cov\) be a countable basis for its topology and let~\(X_n\) be the resulting finite approximations to~\(X\).  Theorem~\ref{the:approximate_E} shows that an arrow in \(\Ecat(X)\) is invertible if and only if its image in \(\Ecat(X_n)\) is invertible for all \(n\in\N\).  (The naturality of the extension in Theorem~\ref{the:approximate_E} implies that the kernel $\varprojlim\nolimits^1 \ldots$ is nilpotent.)  This reduces the general case to the finite case already settled.
\end{proof}

\begin{theorem}
  \label{the:bootstrap_complement_fields}
  Let~$A$ be a separable nuclear \(\Cst\)\nb-algebra with Hausdorff primitive spectrum~$X$.
   Suppose that each two-sided closed ideal of~$A$ is $\KK$-contractible.  Then
  \[
  A\otimes \Cuntza_\infty\otimes \Comp\cong \Cont_0(X)\otimes \Cuntza_2 \otimes \Comp.
  \]
\end{theorem}

\begin{proof} By a result of Fell, $A$ is a continuous $\Cont_0(X)$-algebra with nonzero simple fibres.
  Set $B \defeq \Cont_0(X)\otimes \Cuntza_2$.  Then $0\in \E(X;A,B)$ is an $\E(X)$-equivalence by Theorem~\ref{the:E-iso_finite}.  Theorem~\ref{the:E_versus_KK_Hausdorff} yields $\E_*(X;C,D)\cong\KK_*(X;C,D)$ for $C,D\in\{A,B\}$ because $A$ and~$B$ are nuclear and continuous $\Cont_0(X)$-algebras.  Hence $0\in\KK(X;A,B)$ is a $\KK(X)$-equivalence, and we may apply the main result of~\cite{Kirchberg:Michael} to conclude that
  \(A\otimes \Cuntza_\infty\otimes \Comp\cong B\otimes \Cuntza_\infty \otimes \Comp\).
\end{proof}

\section{The \texorpdfstring{$\E$}{E}-theoretic bootstrap category}
\label{sec:bootstrap_E}

Recall that the bootstrap class~\(\Bootstrap\) in~\(\KKcat\) is the localising subcategory of the triangulated category~\(\KKcat\) that is generated by the object~\(\C\).  Similarly, we define the \(\E\)\nb-theoretic bootstrap class \(\BootstrapE\subseteq \Ecat\) as the localising subcategory of~\(\Ecat\) generated by~\(\C\).  This is the class of all separable \(\Cst\)\nb-algebras~\(A\) for which \(\E_*(A,B)\) fulfills the Universal Coefficient Theorem for all~\(B\).

For a finite topological space~\(X\), a bootstrap class~\(\Bootstrap(X)\) in \(\KKcat(X)\) is defined in~\cite{Meyer-Nest:Bootstrap} along similar lines.  Here we follow a different approach:

\begin{definition}
  \label{def:Bootstrap_E}
  Let \(\BootstrapE(X)\subseteq \Ecat(X)\) for a second countable topological space~\(X\) be the class of all separable \(\Cst\)\nb-algebras~\(A\) over~\(X\) with \(A(U)\in\BootstrapE\) for all \(U\in\Open(X)\).
\end{definition}

Since the functors \(\Ecat(X)\to\Ecat\), \(A\mapsto A(U)\), are triangulated and commute with direct sums and~\(\BootstrapE\) is a localising subcategory of~\(\Ecat\), \(\BootstrapE(X)\) is a localising subcategory of \(\Ecat(X)\).  Furthermore, if \(A\in\BootstrapE(X)\), then \(A(Y)\in\BootstrapE\) for all locally closed subsets \(Y\subseteq X\) because of the extension \(A(U)\into A(V)\prto A(Y)\) with \(Y=V\setminus U\) and suitable open subsets \(U\) and~\(V\) in~\(X\).

\begin{proposition}
  \label{pro:finite_bootstrapE}
  Let~\(X\) be a finite topological space and let~\(A\) be a separable \(\Cst\)\nb-algebra over~\(X\).  Then \(A\in\BootstrapE(X)\) if and only if \(A(\cl{\{x\}})\in\BootstrapE\) for all \(x\in X\).
\end{proposition}

If~\(A\) is tight, that is, the map \(\Prim(A)\to X\) is a homeomorphism, then the \(\Cst\)\nb-algebras \(A(\cl{\{x\}})\) for \(x\in X\) are precisely the \emph{prime quotients} of~\(A\).

\begin{proof}
  Since~\(\BootstrapE\) is triangulated, the class~\(\Good\) of locally closed subsets~\(Y\) of~\(X\) with \(A(Y)\in\BootstrapE\) has the following property: if \(Y\subseteq Z\) and if two of \(Y,Z,Z\setminus Y\) belong to~\(\Good\), then so does the third.  We are going to prove that a set~\(\Good\) of subsets must contain all locally closed subsets if it has this two-out-of-three property and contains all point closures \(\cl{\{x\}}\).  The proof is by induction on the length of the subspace~\(\cl{Y}\), that is, the length of the largest chain \(x_0\prec x_1\prec\dotsb\prec x_\ell\) in the specialisation preorder on the closure~\(\cl{Y}\).  If \(\ell=0\), the subspace~\(Y\) is a set of closed points of~\(X\), and the assertion is easy.

  Let~\(Y\) be a locally closed subset of~\(X\) of length~\(\ell\).  Then \(Y=\cl{Y}\setminus \partial Y\), so that it suffices to prove \(\cl{Y},\partial Y\in\Good\).  Therefore, we may assume without loss of generality that~\(Y\) is closed.  Let \(Z\subseteq Y\) be the set of all open points of~\(Y\).  The difference \(Y\setminus Z\) has length~\(\ell-1\) and is therefore good by induction assumption.  If \(x\in Z\), then the closure \(\cl{\{x\}}\) is good by assumption, and \(\cl{\{x\}} \setminus \{x\}\) is good because its length is at most~\(\ell-1\).  Hence \(\{x\}\) is good for all \(x\in Z\).  Since~\(Z\) is discrete, it follows that~\(Z\) is good.  Hence so is~\(Y\).
\end{proof}

Similarly, if~\(X\) is finite, then \(A\in\BootstrapE(X)\) if and only if \(A(U_x)\in\BootstrapE\) for all \(x\in X\), where~\(U_x\) denotes the minimal open subset of~\(X\) containing~\(x\).

Proposition~\ref{pro:finite_bootstrapE} remains true for some infinite spaces~\(X\) as well.  For instance, let~$X$ be a finite-dimensional, compact, metrisable Hausdorff space.  It is proved in~\cite{Dadarlat:Fiberwise_KK} that a continuous, separable and nuclear $\Cont(X)$-algebra~$A$ lies in the bootstrap class~$\Bootstrap$ if all its fibres $A(x) = A(\cl{\{x\}})$ are in~$\Bootstrap$.  Applying this to all closed subsets of~$X$, we get $A\in\BootstrapE(X)$ under the same assumptions.

For finite spaces~\(X\), we may also describe the bootstrap class in terms of generators.  For \(x\in X\) and a \(\Cst\)\nb-algebra~\(A\), let~\(i_x A\) be~\(A\) viewed as a \(\Cst\)\nb-algebra over~\(X\) concentrated over \(x\in X\), that is, \(i_x(A)(U) = A\) for \(x\in U\) and \(i_x(A)(U) = 0\) for \(x\notin U\).  This \(\Cst\)\nb-algebra over~\(X\) satisfies
\[
\KK_*(X;i_x(A),B) \cong \KK_*\bigl(A,B(U_x)\bigr)
\]
for all~\(B\) by \cite{Meyer-Nest:Bootstrap}*{Proposition 3.13}.  The same argument with \(\E\)\nb-theory instead of \(\KK\)-theory yields
\begin{equation}
  \label{eq:KK_ix}
  \E_*(X;i_x(A),B) \cong \E_*\bigl(A,B(U_x)\bigr)
\end{equation}
for \(x\in X\), a \(\Cst\)\nb-algebra~\(A\) and a \(\Cst\)\nb-algebra~\(B\) over~\(X\).  Here~\(U_x\) denotes the minimal open neighbourhood of~\(x\), which exists because~\(X\) is finite.  Furthermore,
\begin{equation}
  \label{eq:KK_ix_opposite}
  \E_*\bigl(X;A,i_x(B)\bigr)
  \cong \E_*\bigl(A\bigl(\cl{\{x\}}\bigr),B\bigr)
\end{equation}
as in~\cite{Meyer-Nest:Bootstrap}, even for infinite~\(X\), but we will not use this in the following.

\begin{proposition}
  \label{pro:Bootstrap_generators_finite}
  Let~\(X\) be a finite topological space.  Then \(\BootstrapE(X)\) is the localising subcategory of \(\Ecat(X)\) that is generated by~\(i_x\C\) for all \(x\in X\).  The whole category \(\Ecat(X)\) is generated by \(\Cst\)\nb-algebras of the form~\(i_xA\) for separable \(\Cst\)\nb-algebras~\(A\) and \(x\in X\).
\end{proposition}

\begin{proof}
  It is clear that \(i_x\C\in\BootstrapE(X)\) and that \(\BootstrapE(X)\) is localising, so that it contains the localising subcategory generated by~\(i_x\C\) for \(x\in X\).  The same proof as for \cite{Meyer-Nest:Bootstrap}*{Proposition 4.7} shows that a \(\Cst\)\nb-algebra~\(A\) over~\(X\) belongs to the localising subcategory of \(\Ecat(X)\) generated by \(i_x\bigl(A(x)\bigr)\) for all \(x\in X\).  The admissibility assumptions in~\cite{Meyer-Nest:Bootstrap} are only needed for \(\KK\), they become automatic in \(\E\)\nb-theory.  In particular, this shows that \(\Ecat(X)\) is generated by \(\Cst\)\nb-algebras of the form~\(i_xA\), while \(\BootstrapE(X)\) is generated by~\(i_xA\) with \(A\in\BootstrapE\).  Since~\(\BootstrapE\) is generated by~\(\C\), we may replace the set of~\(i_xA\) with \(A\in\BootstrapE(X)\) by~\(i_x\C\) here.
\end{proof}

\begin{theorem}
  \label{the:bootstrap_complement}
  Let~\(X\) be a second countable topological space and let \(A\) and~\(B\) belong to \(\BootstrapE(X)\).  An element in \(\E_*(X;A,B)\) is invertible if and only if it induces invertible maps \(\K_*\bigl(A(U)\bigr) \to \K_*\bigl(B(U)\bigr)\) for all \(U\in\Open(X)\).
\end{theorem}

\begin{proof}
  It is well-known that an element in \(\KK_*(A,B)\) that induces an isomorphism on \(\K\)\nb-theory is invertible in \(\KK\) provided \(A\) and~\(B\) belong to the bootstrap category.  The same argument applies to \(\E\)\nb-theory.  Finally, apply Theorem~\ref{the:E-iso_finite} and the definition of \(\BootstrapE(X)\).
\end{proof}

\section{Comparing \texorpdfstring{$\KK$}{KK}- and \texorpdfstring{$\E$}{E}-theory}
\label{sec:cp_KK}

In the definition of \(\E\)\nb-theory, we may restrict attention to asymptotic morphisms~\(\varphi\) for which the maps~\(\varphi_t\) are all completely positive contractions.  It is shown by Houghton-Larsen and Thomsen \cite{Houghton-Larsen-Thomsen:Universal} that the resulting variant of \(\E\)\nb-theory agrees with Kasparov's \(\KK\).  A corresponding result for equivariant \(\KK\)- and \(\E\)\nb-theory is established by Thomsen in~\cite{Thomsen:Asymptotic_KK}.  It is a routine exercise to show that the same works in our situation.

\begin{definition}
  \label{def:cp_asm}
  Let \(\asm{A,B}^\cp\) denote the space of homotopy classes of \(X\)\nb-equi\-variant, completely positive, linear, contractive asymptotic morphisms~\(\varphi\) from~$A$ to~\(B\), where homotopy is defined using \(X\)\nb-equivariant, completely positive, linear, contractive asymptotic morphisms \(A\to\Cont_\bo\bigl(T,\Cont([0,1],B)\bigr)\).  \(X\)\nb-equivariance means \(\varphi\bigl(A(U)\bigr) \subseteq \Cont_\bo\bigl(T,B(U)\bigr)\) for all \(U\in\Open(X)\).
\end{definition}

The map \(\varphi\colon A\to\Cont_\bo(T,B)\) is an \(X\)\nb-equivariant, completely positive, linear contraction if and only if all the individual maps \(\varphi_t\colon A\to B\) are \(X\)\nb-equivariant, completely positive, linear contractions.

\begin{theorem}
  \label{the:KK_via_asm}
  There is a natural isomorphism
  \[
  \KK_0(X;A,B) \cong \asm{\Cont_0(\R,A)\otimes\Comp, \Cont_0(\R,B)\otimes\Comp}^{\cp}.
  \]
\end{theorem}

\begin{proof}
  Copy the proofs of the corresponding assertions for non-equivariant Kasparov theory and equivariant Kasparov theory for group actions in \cites{Houghton-Larsen-Thomsen:Universal, Thomsen:Asymptotic_KK}.  The main point is to go through the proof of the universal property of \(\E\)\nb-theory and to check that the variant \(\asm{\Cont_0(\R,A)\otimes\Comp, \Cont_0(\R,B)\otimes\Comp}^\cp\) satisfies an analogous universal property, but with exactness only for extensions of \(\Cst\)\nb-algebras over~\(X\) with a completely positive contractive section over~\(X\).  Since \(\KKcat(X)\) satisfies the same universal property, the two theories must be naturally isomorphic.

  Our case is somewhat closer to case of non-equivariant \(\KK\) in~\cite{Houghton-Larsen-Thomsen:Universal} because some issues like Hilbert space representations of groups and equivariance of approximate units do not occur.
\end{proof}

\begin{corollary}
  \label{cor:E_versus_KK}
  Let~\(X\) be a second countable  topological space and let \(A\) be a \(\Cst\)\nb-algebra over~\(X\) which is $\KK(X)$-equivalent to a \(\Cst\)\nb-algebra over~\(X\), \(A^\prime\) such that any extension \(I\into E\prto \Cont_0(\R,A^\prime)\otimes\Comp\) of \(\Cst\)\nb-algebras over~\(X\) has an \(X\)\nb-equivariant completely positive contractive linear section.  Then the canonical map \(\KK_0(X;A,B) \to \E_0(X;A,B)\) is an isomorphism for any \(\Cst\)\nb-algebra~\(B\) over~\(X\).
\end{corollary}

\begin{proof} We may assume that $A=A^\prime$.
  Any asymptotic morphism is equivalent to one with~\(\varphi_0=0\) --~multiply pointwise with a suitable scalar-valued function.  Hence it makes no difference whether we assume this for the definition of \(\asm{A,B}\) and \(\asm{A,B}^\cp\).  An asymptotic morphism from~\(A\) to~\(B\) with \(\varphi_0=0\) generates an extension \(\Cont_0(T,B) \into E\prto A\) with \(E=\varphi(A)+\Cont_0(T,B) \subseteq \Cont_\bo(T,B)\), and two asymptotic morphisms generate the same extension if and only if they are equivalent.  The asymptotic morphism itself is a section for this extension.  The assumption of the corollary therefore implies \(\asm{\Cont_0(\R,A)\otimes\Comp,D}^\cp = \asm{\Cont_0(\R,A)\otimes\Comp,D}\) for all~\(D\).
\end{proof}

\begin{theorem}
  \label{the:E_versus_KK_Hausdorff}
  Let~\(X\) be a second countable locally compact Hausdorff space, let~\(A\) be a nuclear and continuous \(\Cst\)\nb-algebra over~\(X\), and let~\(B\) be any separable \(\Cst\)\nb-algebra over~\(X\).  Then the canonical map \(\KK_0(X;A,B) \to \E_0(X;A,B)\) is an isomorphism.
\end{theorem}

\begin{proof}
  The result follows from \cite{Park-Trout:Representable_E}*{Theorem~4.7}. Alternatively, we may argue that~$A$ is \(\Cont_0(X)\)-nuclear by \cite{Bauval:KK(X)-nuc}*{Theorem~7.2}, so that it satisfies the assumptions of Corollary~\ref{cor:E_versus_KK}.
\end{proof}

\begin{theorem}
  \label{the:E_versus_KK_finite}
  Let~\(X\) be a finite topological space and let \((A,\psi_A)\) and \((B,\psi_B)\) be \(\Cst\)\nb-algebras over~\(X\).  The canonical map
  \[
  \KK_*(X;A,B) \to \E_*(X;A,B)
  \]
  is an isomorphism if~\(A\) belongs to the bootstrap class in \(\KKcat(X)\) defined in~\cite{Meyer-Nest:Bootstrap}.  In particular, this applies if the \(\Cst\)\nb-algebra \(A(X)\) is nuclear.
\end{theorem}

\begin{proof}
  If~\(A\) belongs to the bootstrap class of~\cite{Meyer-Nest:Bootstrap}, then we may compute \(\KK_*(X;A,B)\) by a spectral sequence whose first page only involves the \(\K\)\nb-theory groups of \(A(U)\) and \(B(U)\) for minimal open subsets~\(U\) in~\(X\).  The arguments in~\cite{Meyer-Nest:Bootstrap} only use the universal property of \(\KKcat(X)\) and work equally well for~\(\Ecat(X)\), with some simplifications because we do not have to worry about equivariant completely positive sections of various extensions.  Thus there is an analogous spectral sequence computing \(\E_*(X;A,B)\), and it has the same first page as the spectral sequence computing \(\KK_*(X;A,B)\).  The canonical map \(\KKcat(X)\to\Ecat(X)\) provides a morphism between these spectral sequences, which is an isomorphism on the first page and thus on all later pages.  Hence the two spectral sequences are isomorphic, so that \(\KK_*(X;A,B) \cong \E_*(X;A,B)\).
\end{proof}

\begin{example}
  \label{example:no_excision}
  We exhibit an extension of nuclear \(\Cst\)\nb-algebras over $[0,1]$ which is not excisive for $\KK([0,1];\blank,B)$.  Consider the extension of \(\Cst\)\nb-algebras over $[0,1]$
  \[
  0\to \Cont_0[0,1)\to \Cont[0,1] \xrightarrow{\pi} \C \to 0,
  \]
  where $\pi(f)=f(1)$.  We claim that the mapping cone~$C_\pi$ is not $\KK([0,1])$-equivalent to $\ker(\pi)=\Cont_0[0,1)$ and that
  \[
  \KK([0,1];S\C, \Cont_0[0,1))\neq \E([0,1];S\C, \Cont_0[0,1)).
  \]
  Here~$S\C$ is regarded as a $\Cont[0,1]$-algebra via the multiplication $f \cdot g=f(1)g$ for $f\in \Cont[0,1]$ and $g\in S\C$.  Let us address first the second part of the claim.  It is convenient to work with asymptotic morphisms parametrised by $t\in [0,1)$.  For each such~$t$ consider the map $\nu_t\colon [0,1]\to [0,1]$,
  \[
  \nu_t(s)=
  \begin{cases}
    0 &\text{if $0\le s <t$,}\\
    \frac{s-t}{1-t} & \text{if $t\le s \le 1$.}
  \end{cases}
  \]
  Define a continuous family of \Star{}homomorphisms $\varphi_t\colon S\C\to \Cont_0[0,1)$, $t\in [0,1)$ by $\varphi_t (\exp{(2\pi\ima s)}-1) \defeq \exp{(2\pi\ima \nu_t(s))}-1$.  It is easily verified that the asymptotic homomorphism~$(\varphi_t)$ is asymptotically $[0,1]$-equivariant since $\exp{(2\pi\ima \nu_t(s))}-1$ is suported on $[t,1)$.  Set $A=S\C$ and $B=\Cont_0[0,1)$.  We observe that the class of $(\varphi_t)$ in $\E([0,1];A, B)$ is non-zero since its image in $\Hom\bigl(\K_1(A(0,1)), \K_1(B(0,1))\bigr) \cong \Hom(\Z, \Z)$ is equal to $\id_\Z$.  On the other hand, $\KK_*([0,1];A, B)=\KK_*\bigl(S\C,\bigcap_n B((1-1/n,1])\bigr)=\KK_*(S\C,\{0\})=0$, by \cite{Meyer-Nest:Bootstrap}*{Proposition 3.13}.

  Let us verify now the first part of the claim.  The Puppe sequence for $\KK([0,1];\blank,B)$ associated to the map~$\pi$ yields $\KK([0,1],C_\pi,B)=0$ since $\KK_*([0,1],\Cont[0,1],B)=\KK_*(\C,B)=0$ and $\KK_*([0,1];\C, B)=0$ as argued above.  At the same time, $\KK_*([0,1];B, B)\neq 0$ since the natural map $\KK_*([0,1];B, B)\to \Hom(\K_1(B(0,1)), \K_1(B(0,1))\cong \Z$ sends $[\id_B]$ to~$1$.
 \end{example}

\section{A universal coefficient theorem for \texorpdfstring{$\Cst$}{C*}-algebras over totally disconnected spaces}
\label{sec:Cstar_Cantor}

In this section, we study \(\Cst\)\nb-algebras over a totally disconnected compact metrisable space~$X$.  Our goal is to construct a Universal Coefficient Theorem that computes \(\E_*(X;A,B)\) for \(A,B\in\BootstrapE(X)\).  For this purpose, we use filtrated \(\K\)\nb-theory \emph{with coefficients} and obtain a Universal Coefficient exact sequence that generalises the Multicoefficient Theorem of~\cite{Dadarlat-Loring:Multicoefficient}.  In order to explain the key role of filtrated \(\K\)\nb-theory with coefficients, we also revisit an example from~\cite{Dadarlat-Eilers:Bockstein} showing that the spectral sequence generated by filtrated \(\K\)\nb-theory does not degenerate to an exact sequence.

In this section, all \(\Cst\)\nb-algebras are assumed separable and all groups countable.

Let \(\Primep\subseteq\N\) be the set consisting of~\(0\) and all prime powers. The relevance of the set~\(\Primep\) in the Universal Multicoefficient Theorem is that the groups~\(\Z/p\) for \(p\in\Primep\) are exactly the indecomposable Abelian groups.

For $p\in\Primep$ let~$\III_p$ be the mapping cone of the unital \Star{}homomorphism $\C\to \Mat_p(\C)$.  For $p=0$, we let $\III_0\defeq\C$.  It is convenient to denote~$\III_p$ by~$\III_p^0$ and its suspension~$S\III_p$ by~$\III_p^1$.  Then for a \(\Cst\)\nb-algebra~$A$:
\[
\K_i(A;\Z/p)\defeq\KK_i(\III_p,A)\cong \KK(\III_p^i,A), \quad i=0,1.
\]

Let us set $\III\defeq \bigoplus_{p\in\Primep} \III_p$ and consider the ring $\KK_*(\III,\III)$ with multiplication given by the Kasparov product.  The non-unital subring
\[
\Lambda=\bigoplus_{p,q\in\Primep} \KK_*(\III_p,\III_q)
\]
of $\KK_*(\III,\III)$ is called the ring of \emph{B\"ockstein operations}.  It consists of matrices indexed by $\Primep\times\Primep$ with only finitely many non-zero entries $\lambda_{pq}\in \KK_*(\III_p,\III_q)$.  The Kasparov product
\[
\KK_*(\III_p,\III_q)\times \KK_*(\III_q,A) \to \KK_*(\III_p,A)
\]
induces a natural $\Lambda$\nb-module structure on the $\Z/2\times \Primep$-graded group
\[
\Kcoef(A)=\bigoplus_{p\in \Primep} \K_*(A;\Z/p).
\]
The $\KK$-class~$x_p^i$ of $\id_{\III_p^i}$ generates the group $\K_i(\III_p^i, \Z/p)\cong \KK(\III_p^i,\III_p^i)$.  We shall work with $\Z/2\times \Primep$-graded $\Lambda$\nb-modules $M=(M_p^i)$ such that for $\lambda\in \KK_j(\III_q,\III_k)$ and $m\in M_p^i$, $\lambda m\in M_q^{j+i}$ if $k=p$ and $\lambda m=0$ if $k\neq p$.  We also ask that~$x_p^i$ acts as the identity automorphism on~$M_p^i$.  In particular, this implies that $p M_p^i=0$.  These assumptions are modelled on the case $M=\Kcoef(A)$ where $M_p^i=\KK_*(\III_p^i,A)$.

\begin{definition}
  \label{def:Cantor_basic_modules}
  A $\Lambda$\nb-module isomorphic to $\Kcoef(\III_p^i)$ for some $(i,p)\in \Z/2\times \Primep$ is called \emph{basic}.
\end{definition}

\begin{lemma}
  \label{lemma:Cantor_Lambda_proj}
  For all $(i,p)\in \Z/2\times \Primep$, $\Kcoef(\III_p^i)=\Lambda \cdot x_p^i$.  The basic $\Lambda$\nb-modules are projective in the category of $\Z/2\times \Primep$-graded $\Lambda$\nb-modules.
\end{lemma}

\begin{proof}
  The first part follows because $\KK_*(\III_p,\III_p)\cong \KK_*(\III_p^i,\III_p^i)$ and $x_p^i=[\id_{\III_p^i}]$ is idempotent.  For the second part we observe that if $\lambda x_p^i=0$ for some $\lambda\in \KK_*(\III_q,\III_k)$ then either $k\neq p$ or $\lambda=0$.  This shows that if $\pi\colon B\to C$ is a surjective morphism of $\Lambda$\nb-modules, then any morphism $\varphi\colon \Lambda x_p^i \to C$ lifts to a morphism $\Phi\colon \Lambda x_p^i \to B$ defined by $\Phi(\lambda x_p^i)=\lambda b_p^i$, $\lambda\in\Lambda$, where $b_p^i$ is some lifting of $\varphi(x_p^i)$.
\end{proof}

We give a very concise proof of the following result from~\cite{Dadarlat-Loring:Multicoefficient}.

\begin{proposition}
  \label{prop:Cantor_Lambda_fin_gen_K}
  Let \(A\) and~\(B\) be separable \(\Cst\)\nb-algebras and suppose that~$A$ is in the bootstrap class~$\Bootstrap$ with $\K_*(A)$ finitely generated.  Then $\KK(A,B) \cong \Hom_\Lambda\bigl(\Kcoef(A),\Kcoef(B)\bigr)$.
\end{proposition}

\begin{proof}
  Both sides are additive in the first variable.  Thus by the UCT we may assume that $A=\III_p^i$ for some $(i,p)\in \Z/2\times \Primep$.  Let us observe that any element $h\in \Hom_\Lambda\bigl(\Lambda x_p^i,\Kcoef(B)\bigr)$ is completely determined by $h(x_p^i)\in \K_i(B;\Z/p)\cong\KK(\III_p^i,B)$.  Moreover, the image of $h(x_p^i)$ under the map $\KK(\III_p^i,B) \to \Hom_\Lambda\bigl(\Kcoef(\III_p^i),\Kcoef(B)\bigr)$ is precisely~$h$.  Indeed, the Kasparov product $\KK(\III_p^i,\III_p^i)\times\KK(\III_p^i,B)\to \KK(\III_p^i,B)$ gives $[\id_{\III_p^i}]\times\alpha=\alpha$.
\end{proof}

If~$A$ is a separable \(\Cst\)\nb-algebra over a zero-dimensional space~$X$, then $\Kcoef(A)$ has a natural structure of module over the ring \(\Cont(X,\Lambda)\) of locally constant functions from~$X$ to~$\Lambda$.  This is easily seen by observing that $A\cong \bigoplus_{k=1}^n A(U_k)$ for any clopen partition $(U_k)_{k=1}^n$ of~$X$.  A \(\Cst\)\nb-algebra over~$X$ is called \emph{elementary} if it is isomorphic to \(\bigoplus_{k=1}^n \Cont(U_k,A_k)\), where $(U_k)_{k=1}^n$ is a clopen partition of~$X$, each~$A_k$ is a separable \(\Cst\)\nb-algebra in the bootstrap class, and $\K_*(A_k)$ is finitely generated.  If~$A$ is elementary, then the \(\Cont(X,\Lambda)\)-module \(\Kcoef(A)\) is isomorphic to \(\bigoplus_{k=1}^n \Cont(U_k,\Kcoef(A_k))\).  Since $\K_*(A_k)$ is finitely generated, it follows from the UCT that~$A_k$ is $\KK$-equivalent to a finite direct sum of~$\III_p^i$s, so that~$\Kcoef(A_k)$ is $\Lambda$\nb-projective by Lemma~\ref{lemma:Cantor_Lambda_proj}.  It follows easily that the \(\Cont(X,\Lambda)\)-module~\(\Kcoef(A_k)\) is projective.

\begin{lemma}
  \label{lemma:from_Pext_to_lim_one_Cantor}
  Suppose that~$M$ is isomorphic to the inductive limit of an inductive system~$(M_j)$ of projective \(\Cont(X,\Lambda)\)-modules.  Then for any \(\Cont(X,\Lambda)\)-module~\(N\) there is a natural isomorphism
  \[
  \varprojlim\nolimits^1 \Hom_{\Cont(X,\Lambda)}(M_j,N) \cong \Ext_{\Cont(X,\Lambda)}(M,N).
  \]
\end{lemma}

\begin{proof}
  Set $R=\Cont(X,\Lambda)$.  The extension
  \[
  0 \to \bigoplus_{j\in\N} M_j \xrightarrow{\Id-S}
  \bigoplus_{j\in\N} M_j \to M \to 0,
  \]
  where~$S$ is the natural shift map, is a projective resolution of~$M$.  Since $\bigoplus_{j\in\N} M_j$ is projective, we have an exact sequence
  \[
  \Hom_R\Bigl(\bigoplus_{j\in\N} M_j,N\Bigr) \xrightarrow{(\id-S)^*}
  \Hom_R\Bigl(\bigoplus_{j\in\N} M_j,N\Bigr) \to \Ext_{R}(M,N) \to 0,
  \]
  where the first map identifies with the first map of the exact sequence
  \[
  \prod_{j\in\N} \Hom_R(M_j,N) \to
  \prod_{j\in\N} \Hom_R(M_j,N) \to \varprojlim\nolimits^1 \Hom_R(M_j,N) \to 0
  \]
  that defines $\varprojlim\nolimits^1$.  Thus the two maps have isomorphic cokernels.
\end{proof}

\begin{proposition}
  \label{prop:Write_as_limit_Cantor}
  Let~$A$ be a separable nuclear continuous \(\Cst\)\nb-algebra over a totally disconnected compact metrisable space~\(X\).  Suppose that each fibre of~$A$ belongs to the bootstrap class~\(\Bootstrap\).  Then~$A$ is $\KK(X)$-equivalent to the inductive limit of an inductive system of elementary \(\Cont(X)\)-algebras.
\end{proposition}

\begin{proof}
  \cite{Dadarlat:Fiberwise_KK}*{Theorem~2.5} shows that~$A$ is $\KK(X)$-equivalent to a unital continuous \(\Cont(X)\)-algebra~$A^\sharp$ whose fibres are Kirchberg algebras.  Thus we may assume that $A=A^\sharp$.  By \cite{Dadarlat-Pasnicu:Continuous_fields}*{Theorem~3.6}, there is a sequence $(A_n)_{n=1}^\infty$ of elementary unital \(\Cont(X)\)-subalgebras of~$A$ which is exhausting~$A$ in the sense that for every finite subset~$F$ of~$A$, $\lim_{n\to\infty} \operatorname{dist}(F,A_n)=0$.  Since~$A_n$ is locally trivial and its fibres are weakly semiprojective (\cite{Dadarlat:cont_fields_fd_spaces}*{Section~3}) each inclusion map $\gamma_n\colon A_n\hookrightarrow A$ can be perturbed to some $\Cont(X)$-linear unital \Star{}monomorphism $\gamma_{n,n+k}\colon A_n \to A_{n+k}$ with $\norm{\gamma_n(a)-\gamma_{n,n+k}(a)}<1/2^n$ for~$a$ in a prescribed finite subset of~$A_n$.  It follows that after passing to a subsequence of~$(A_n)$ we can represent~$A$ as the inductive limit of a system $(A_{n_k},\gamma_{n_k,n_{k+1}})$ of elementary $\Cont(X)$-algebras.
\end{proof}

\begin{lemma}
  \label{lemma:UMCT_Lambda_finite_Cantor}
  Let \(A\) and~\(B\) be separable \(\Cont(X)\)-algebras over a totally disconnected compact metrisable space~\(X\) and suppose that~$A$ is elementary.  Then $\KK(X;A,B) \cong \Hom_{\Cont(X,\Lambda)}\bigl(\Kcoef(A),\Kcoef(B)\bigr)$.
\end{lemma}

\begin{proof}
  Write $A=\bigoplus_{i=1}^k \Cont(U_i, D_i)$ where $U_1,\dotsc,U_k$ is a clopen partition of~$X$ and each~$D_i$ is in the bootstrap class with $\K_*(D_i)$ finitely generated.  We have $\KK(X;A,B)\cong \bigoplus_{i=1}^k \KK(U_i;A(U_i),B(U_i))$ and
  \[
  \Hom_{\Cont(X,\Lambda)}\bigl(\Kcoef(A),\Kcoef(B)\bigr)\cong \bigoplus_{i=1}^k \Hom_{\Cont(U_i,\Lambda)}\bigl(\Kcoef(A(U_i)),\Kcoef(B(U_i))\bigr).
  \]
  Thus we may assume that $A=\Cont(X,D)$.  In this case, the assertion follows from the commutative diagram
  \[
  \xymatrix{
    \KK(X;\Cont(X,D), B)\ar[r]\ar[d]_\cong&\Hom_{\Cont(X,\Lambda)}\bigl(\Kcoef(\Cont(X,D)),\Kcoef(B)\bigr)\ar[d]^\cong\\
    \KK(D,B)\ar[r] ^\cong&\Hom_\Lambda(\Kcoef(D),\Kcoef(B))}
  \]
  The bottom horizontal map of the diagram is bijective by Proposition~\ref{prop:Cantor_Lambda_fin_gen_K}, the left vertical may by Lemma~\ref{lem:E_compact_Hausdorff}.  The right vertical map is bijective because
  \begin{multline*}
    \Kcoef\bigl(\Cont(X,D)\bigr)
    \cong \Cont\bigl(X,\Kcoef(D)\bigr)
    \cong \Cont(X,\Z) \otimes \Kcoef(D)
    \\\cong \Cont(X,\Z) \otimes \Lambda \otimes_\Lambda \Kcoef(D)
    \cong \Cont(X,\Lambda) \otimes_\Lambda \Kcoef(D)
  \end{multline*}
  and
  \[
  \Hom_{\Cont(X,\Lambda)}\bigl(\Cont(X,\Lambda)\otimes_\Lambda \Kcoef(D), \Kcoef(B)\bigr)
  \cong \Hom_\Lambda\bigl(\Kcoef(D), \Kcoef(B)\bigr).\qedhere
  \]
\end{proof}

\begin{lemma}
  \label{lemma:C(X)alg_equivto_field}
  Any separable \(\Cont(X)\)-algebra over a totally disconnected compact metrisable space~\(X\) is isomorphic to the inductive limit of a sequence of locally trivial separable \(\Cont(X)\)-algebras.
\end{lemma}

\begin{proof}
  Let~$A$ be a separable \(\Cont(X)\)-algebra over~$X$.  If~$\cov$ is a finite clopen cover of~$X$ we denote by~$A_\cov$ the locally trivial continuous \(\Cont(X)\)-algebra $\bigoplus_{U\in\cov} \Cont(U)\otimes A(U)$.  For each $x\in U$ the fibre $A_\cov(x)$ is $A(U)$.  There is a natural morphism of \(\Cont(X)\)-algebras $\alpha_\cov\colon A_\cov \to A$ which maps $(f_U \otimes a_U)_{U\in\cov}$ to $\sum_{U\in \cov} f_U a_U$.

  If~$V$ is a closed subset of~$U$ we have a natural restriction homomorphism $\Cont(U)\otimes A(U)\to \Cont(V)\otimes A(V)$, which maps $f\otimes a$ to $f|_V\otimes \pi_V(a)$.  Therefore, if~$\mathcal{V}$ is a finite clopen cover of~$X$ which refines~$\cov$, there is a natural morphism of \(\Cont(X)\)-algebras $\alpha_\cov^\mathcal{V}\colon A_\cov \to A_\mathcal{V}$ such that $\alpha_\mathcal{V}\circ \alpha_\cov^\mathcal{V}=\alpha_\cov$.

  Let $(\cov_n)_n$ be an infinite sequence of finite clopen covers of~$X$, with $\cov_{n+1}$ refining~$\cov_n$, and such that $\operatorname{diam}(\cov_n)\to 0$ with respect to some metric inducing the topology of~$X$.  Set $A_n=A_{\cov_n}$, $\alpha_n=\alpha_{\cov_n}$ and $\alpha_n^m=\alpha_{\cov_n}^{\cov_m}$.  We claim that the natural morphism $\varinjlim (A_n,\alpha_n^m) \to A$ is an isomorphism.  This morphism is surjective since each~$\alpha_n$ is surjective.  To prove its injectivity, it suffices to show that if $F\in A_n$ satisfies $\alpha_n(F)=0$, then for any $\varepsilon>0$ there is $m>n$ such that $\norm{\alpha_n^m(F)}\le\varepsilon$.  By localising at each element of~$\cov_n$, we may assume that $A_n=\Cont(X)\otimes A(X)$ and regard~$F$ as a continuous function $F\colon X\to A(X)$.  Since~$F$ is continuous, each $x\in X$ has a neighbourhood~$V_x$ such that $\norm{F(x)-F(y)}<\varepsilon/2$ for all $y\in V_x$.  Since $A(X)$ is a $\Cont(X)$-algebra, for each $a\in A(X)$, the map $x\mapsto \norm{\pi_x(a)}$ is upper semi-continuous.  The assumption $\alpha_n(F)=0$ implies that $\pi_x(F(x))=0$ for all $x\in X$.  Thus, after shrinking each~$V_x$ if necessary, we may arrange that $\norm{\pi_z(F(x))}<\varepsilon/2$ for all $z\in V_x$.  It follows that for any $y,z \in V_x$,
  \[
  \norm{\pi_z(F(y))}\le \norm{\pi_z(F(y)-F(x))} + \norm{\pi_z(F(x))}<\varepsilon.
  \]
  Extract now a finite cover $V_{x_1},\dotsc,V_{x_r}$ of~$X$.  Since $\operatorname{diam}(\cov_m)\to 0$ there is $m>n$ such that each element of~$\cov_m$ is contained in some~$V_{x_i}$.  It follows that $\norm{\alpha_n^m(F)}\le \varepsilon$.
\end{proof}

\begin{proposition}
  \label{prop:C(X)alg_equivto_field}
  Any separable \(\Cont(X)\)-algebra over a totally disconnected compact metrisable space~\(X\) is $\E(X)$-equivalent to a continuous separable \(\Cont(X)\)-algebra.
\end{proposition}

\begin{proof}
  For a given \(\Cont(X)\)-algebra~$A$, let $(A_n,\alpha_n^m)$ be the corresponding inductive system constructed as in the proof of the previous lemma.  Let
  \[
  T(A_n,\alpha_n^m)= \biggl\{ (f_n)\in \bigoplus_{n\in \N} \Cont([n,n+1],A_n) :
  f_{n+1}(n+1)=\alpha_n^{n+1}(f_n(n+1)) \biggr\}
  \]
  be the associated mapping telescope.  Since the mapping telescope construction is functorial, there is a natural \(\Cont(X)\)-linear \Star{}homomorphism
  \[
  \alpha\colon T(A_n,\alpha_m^n)\to T(A,\id_A)\cong SA.
  \]

  Arguing as in the paragraphs following the proof of \cite{Meyer-Nest:BC}*{Proposition 2.6}, it follows that~$\alpha$ is an $\E(X)$-equivalence.  Indeed, let $\tilde{T}(A_m,\alpha_m^n)$ be the variant of $T(A_m,\alpha_m^n)$ where we require $\lim_{t\to\infty} \alpha_m(f_m(t))$ to exist in~$A$ instead of $\lim f_m(t)=0$.  The algebra $\tilde{T}(A_m,\alpha_m^n)$ is contractible over~$X$ in a natural way.  There is a commutative diagram
  \[
  \xymatrix{
    0 \ar[r] &
    T(A_m,\alpha_m^n) \ar[r] \ar[d]_{\alpha} &
    \tilde{T}(A_m,\alpha_m^n) \ar[r] \ar[d] &
    A \ar[r] \ar@{=}[d] &
    0 \\
    0 \ar[r] &
    T(A,\id) \ar[r] &
    \tilde{T}(A,\id) \ar[r] &
    A \ar[r] &
    0,
  }
  \]
  whose rows are short exact sequences.  Since the algebras in the middle are contractible, it follows that~$\alpha$ induces an $\E(X)$-equivalence.  We conclude by observing that $T(A_m,\alpha_m^n)$ is a continuous \(\Cont(X)\)-algebra since it is a \(\Cont(X)\)-subalgebra of a direct sum of continuous \(\Cont(X)\)-algebras.
\end{proof}

\begin{proposition}
  \label{pro:bootstrap_fibrewise_td}
  A separable and nuclear \(\Cont(X)\)-algebra~$A$ over a totally disconnected compact metrisable space~\(X\) belongs to the bootstrap class~$\BootstrapE(X)$ if and only if all its fibres are in the bootstrap call $\BootstrapE$.
\end{proposition}

\begin{proof}
  By Propositions \ref{prop:C(X)alg_equivto_field} and~\ref{pro:E_Hausdorff_compare}, we may assume that~$A$ is a continuous \(\Cont(X)\)-algebra.  By a result of~\cite{Dadarlat:Fiberwise_KK} a separable nuclear continuous \(\Cont(X)\)-algebra over a finite-dimensional compact metrisable space~$X$ belongs to~$\Bootstrap$ if and only if all its fibres belong to~$\Bootstrap$.  This concludes the proof, since a nuclear \(\Cst\)\nb-algebra belongs to $\Bootstrap$ if and only if it belongs to $\BootstrapE$.
\end{proof}

\begin{proposition}
  \label{prop:E_X-trick}
  Let~$A$ be a separable \(\Cont(X)\)-algebra over a totally disconnected compact metrisable space~\(X\).  If \(A(U)\) is $\E$\nb-equivalent to a separable nuclear C*-algebra
  for each clopen set $U\subset X$, then~$A$ is $\E(X)$-equivalent to a separable, continuous, nuclear \(\Cont(X)\)-algebra.
\end{proposition}

\begin{proof} The proposition applies for instance when
  ~$A$ belongs to the bootstrap class~$\BootstrapE(X)$. 
  It was shown in \cite{Dadarlat:Fiberwise_KK}*{Lemma 2.2} that~$A$ is $\KK(X)$-equivalent to a $\Cont(X)$-algebra~$A'$ such that $A'\otimes\Cuntza_\infty\otimes\Comp\cong A'$ and that~$A'$ contains a full projection.  Thus we may assume that~$A$ itself has these properties.  Let $(A_n,\alpha_n^m)$ be the inductive system constructed in the proof of Lemma~\ref{lemma:C(X)alg_equivto_field}, that is, $A_n$ is of the form $\bigoplus_{k=1}^{r(n)} \Cont(U_k)\otimes A(U_k)$ with a partition into clopen sets~$U_k$.  It is clear that $A(U_k)\cong A(U_k)\otimes\Cuntza_\infty \otimes \Comp$ and that $A(U_k)$ contains a full projections.  By assumption, each \(\Cst\)\nb-algebra~$A(U_k)$ is $\E$\nb-equivalent to some nuclear separable \(\Cst\)\nb-algebra and hence it is $\E$\nb-equivalent to some stable Kirchberg algebra~$D_k$.  For each~$k$, Kirchberg's Classification Theorem \cite{Rordam-Stormer:Classification_Entropy}*{Theorem 8.3.3} yields a \Star{}homomorphism $\eta_k\colon D_k\to A(U_k)$ which lifts the given $\E$\nb-equivalence.  Moreover, we may arrange that~$\eta_k$ decomposes as $\eta_k=\mu_k\oplus \theta_k$, where~$\theta_k$ is a full \Star{}monomorphism that factors through the stable Cuntz algebra $\Cuntza_2\otimes\Comp$.  Extending the~$\eta_k$ by $\Cont(X)$-linearity and taking their direct sum, we get a \(\Cont(X)\)-linear monomorphism $\varphi_n\colon B_n \to A_n$, where $B_n\defeq \bigoplus_{k=1}^{r(n)} \Cont(U_k)\otimes D_k$.  Moreover, each~$\varphi_n$ induces an equivalence in $\Ecat(X)$.  Another application of \cite{Rordam-Stormer:Classification_Entropy}*{Theorem 8.3.3} yields $\Cont(X)$-linear \Star{}monomorphisms $\beta_n^{n+1}\colon B_n\to B_{n+1}$ such that for each~$n$ the diagram
  \[
  \xymatrix{
    A_n\ar[r]^{\alpha_n^{n+1}}& A_{n+1}\\
    B_n\ar[r]_{\beta_n^{n+1}}\ar[u]^{\varphi_n}& B_{n+1}\ar[u]_{\varphi_{n+1}}
  }
  \]
  commutes in $\Ecat(X)$ and hence in the category $\KKcat(X)$ (since each~$D_k$ is nuclear).  The uniqueness part of \cite{Rordam-Stormer:Classification_Entropy}*{Theorem 8.3.3} shows that we may arrange that the diagram above commutes up to unitary homotopy.  By \cite{Dadarlat:Shape_theory}*{Section 2} this gives a \(\Cont(X)\)-linear \Star{}homomorphism $\varphi\colon B \to \Cont_b(T,A)/\Cont_0(T,A)$, where~$B$ is the limit of the inductive system $(B_n,\beta_n^{n+1})$, such that the diagram
  \[
  \xymatrix{
    A_n\ar[r]& A\\
    B_n\ar[r]\ar[u]^{\varphi_n}& B\ar[u]_\varphi
  }
  \]
  commutes in $\Ecat(X)$.  By Proposition~\ref{prop:Milnor_regular}, for any separable \(\Cont(X)\)-algebra~$D$ there is a commutative diagram with exact rows
  \[
  \xymatrix{
    \varprojlim\nolimits^1 \E_1(X;A_i,D)\ \ar@{>->}[r]\ar[d]^{\varphi_n^*}&
    \E(X;A,D) \ar@{>>}[r]\ar[d]^{\varphi^*}&
    \varprojlim \E(X;A_i,D)\ar[d]^{\varphi_n^*}\\
    \varprojlim\nolimits^1 \E_1(X;B_i,D)\ \ar@{>->}[r]&
    \E(X;B,D) \ar@{>>}[r]&
    \varprojlim \E(X;B_i,D).
  }
  \]
  Since the maps~$\varphi_n^*$ are bijective by construction, we conclude that~$A$ is $\E(X)$-equivalent to the nuclear continuous \(\Cont(X)\)-algebra~$B$.
\end{proof}

\begin{theorem}
  \label{the:UMCTgeneral_Lambda_Cantor}
  Let \(A\) and~\(B\) be separable \(\Cont(X)\)-algebras over a totally disconnected compact metrisable space~\(X\).  If~$A$ is in the bootstrap class~$\BootstrapE(X)$, then there is an exact sequence
  \[
  \Ext_{\Cont(X,\Lambda)}\bigl(\Kcoef(A),\Kcoef(SB)\bigr) \into
  \E(X;A,B) \prto \Hom_{\Cont(X,\Lambda)}\bigl(\Kcoef(A),\Kcoef(B)\bigr).
  \]
\end{theorem}

\begin{proof}
  By Proposition~\ref{prop:E_X-trick} we may assume that~$A$ is a continuous nuclear \(\Cont(X)\)-algebra with all its fibres in the bootstrap class~$\Bootstrap$.  Then \(\E(X;A,B) \cong \KK(X;A,B)\) by Theorem~\ref{the:E_versus_KK_Hausdorff}.  By Proposition~\ref{prop:Write_as_limit_Cantor} we may also assume $A\cong \varinjlim A_n$ for an increasing sequence~$(A_n)_{n=1}^\infty$ of elementary \(\Cst\)\nb-subalgebras of~$A$.  Then we can apply the $\varprojlim\nolimits^1$-sequence for nuclear continuous \(\Cont(X)\)-algebras and $\KK(X;\blank,\blank)$ to obtain the following exact sequence:
  \[
  \varprojlim\nolimits^1\KK_1(X;A_n,B)\into \KK(X;A,B)\prto \varprojlim \KK(X;A_n,B).
  \]
  By Lemma~\ref{lemma:UMCT_Lambda_finite_Cantor}
  \begin{multline*}
    \varprojlim \KK(X;A_n,B)
    \cong \varprojlim \Hom_{\Cont(X,\Lambda)}\bigl(\Kcoef(A_n),\Kcoef(B)\bigr)
    \\\cong \Hom_{\Cont(X,\Lambda)}\bigl(\varinjlim \Kcoef(A_n),\Kcoef(B)\bigr)
    \cong \Hom_{\Cont(X,\Lambda)}\bigl(\Kcoef(A),\Kcoef(B)\bigr).
  \end{multline*}

  Using again Lemma~\ref{lemma:UMCT_Lambda_finite_Cantor} and Lemma~\ref{lemma:from_Pext_to_lim_one_Cantor}, we get
  \begin{align*}
    \varprojlim\nolimits^1\KK_{1}(X;A_{n},B)
    &\cong \varprojlim\nolimits^1 \Hom_{\Cont(X,\Lambda)}\bigl(\Kcoef(A_n),\Kcoef(SB)\bigr)
    \\&\cong \Ext_{\Cont(X,\Lambda)}\bigl(\Kcoef(A),\Kcoef(SB)\bigr).\qedhere
  \end{align*}
\end{proof}

\begin{remark}
  If~$A$ is a separable nuclear continuous \(\Cont(X)\)-algebra with all the fibres in~$\Bootstrap$, then $A\in\BootstrapE(X)$, and Theorem~\ref{the:E_versus_KK_Hausdorff} shows that the exact sequence from Theorem~\ref{the:UMCTgeneral_Lambda_Cantor} holds with $\KK(X;A,B)$ replacing $\E(X;A,B)$.
\end{remark}

For abelian groups $G$ and~$H$, $\PExt_\Z(G,H)$ denotes the subgroup of $\Ext_\Z(G,H)$ generated by pure extensions, that is, extensions $H\into E \prto G$ whose restrictions to all finitely generated subgroups of~$G$ split.  Theorem~\ref{the:UMCTgeneral_Lambda_Cantor} is a generalisation of the main result of~\cite{Dadarlat-Loring:Multicoefficient}, which corresponds to the case when~$X$ reduces to a point.

\begin{proposition}
  \label{from_Ext_to_Pext}
  Let $A$ and~$B$ be separable \(\Cst\)\nb-algebras.  If $A\in\Bootstrap$, there is a natural isomorphism $\Ext_\Lambda\bigl(\Kcoef(A),\Kcoef(B)\bigr)\cong \PExt_\Z\bigl(\K_*(A),\K_*(B)\bigr)$.
\end{proposition}

\begin{proof}
  Consider the natural restriction map
  \[
  \eta\colon \Ext_\Lambda\bigl(\Kcoef(A),\Kcoef(B)\bigr)\to \Ext_\Z\bigl(\K_*(A),\K_*(B)\bigr).
  \]
  Let $\Kcoef(B) \into \underline M \prto \Kcoef(A)$ be an extension of $\Lambda$\nb-modules.  We claim that its $\eta$\nb-image $\K_*(B) \into M_* \prto \K_*(A)$ is pure.  Purity follows if any element~\(x\) in \(\K_i(A)\) of order \(n\in\Primep_{\ge1}\) lifts to an element in~\(M_i\) of the same order.  Since~\(x\) has order~\(n\), there is an element \(y\in \K_{i+1}(A;\Z/n)\) such that $\beta_n(y)=x$, because of the exactness of the sequence
  \[
  \K_{i+1}(A;\Z/n)\xrightarrow{\beta_n} \K_i(A)\xrightarrow{n} \K_i(A),
  \]
  where $\beta_n \in \Lambda$.  Let \(\hat{y}\in M_n^{i+1}\) be a lifting of~$y$.  Then the image \(\hat{x}\defeq \beta_n(\hat{y})\in M_0^i\) of~\(\hat{y}\) is a lifting of~$x$ of order~$n$.  Thus the image of~$\eta$ is contained in $\PExt_\Z\bigl(\K_*(A),\K_*(B)\bigr)$.

  Conversely, if $\K_*(B) \into G_* \prto \K_*(A)$ is a pure extension of $\Z/2$-graded abelian groups, then the UCT provides a separable \(\Cst\)\nb-algebra~$E$ and an extension of \(\Cst\)\nb-algebras $B\otimes K \into E \prto A$ such that $\K_*(B) \into \K_*(E) \prto \K_*(A)$ is isomorphic to the given extension.  We claim that $\Kcoef(B) \into \Kcoef(E) \prto \Kcoef(A)$ is an extension of $\Lambda$\nb-modules.  Purity yields extensions
  \[
  \Tor_q(\K_*(B), \Z/n) \into
  \Tor_q(\K_*(E), \Z/n) \prto
  \Tor_q(\K_*(A), \Z/n)
  \]
  for any $n\in \Primep$ and for \(q=0,1\).  Furthermore, there is a natural extension
  \[
  \Tor_0(\K_*(A),\Z/n) \into \K_*(A;\Z/n) \prto \Tor_1(\K_*(A),\Z/n),
  \]
  and the same for \(E\) and~\(B\).  Now a diagram chase shows that $\K_*(B;\Z/n) \into \K_*(E;\Z/n) \prto \K_*(A;\Z/n)$ is an extension.

  Having identified the image of~$\eta$ as $\PExt_\Z\bigl(\K_*(A),\K_*(B)\bigr)$, it remains to show that~$\eta$ is injective.  We may assume that~$A$ is nuclear.  Suppose that the extension $\K_*(B) \into \K_*(E) \prto \K_*(A)$ splits.  By the UCT, the class of the extension $B\otimes \Comp \into E \prto A$ in $KK_1(A,B)$ is zero.  It follows that the extension $B\otimes \Comp \into E \prto A$ is stably split, so that the extension $\Kcoef(B) \into \Kcoef(E) \prto \Kcoef(A)$ is trivial.
\end{proof}

The following example adapted from~\cite{Dadarlat-Eilers:Bockstein} shows that the map $\E(X;A,B) \to \Hom_{\Cont(X, \Z)}\bigl(\K_*(A),\K_*(B)\bigr)$ is not always surjective.

\begin{example}
  \label{exa:Filtrated_K_insufficient}
  Let \(X=\N\cup \{\infty\}\) be the one-point compactification of~\(\N\).  We shall exhibit two separable continuous $\Cont(X)$-algebras $E_k$ and~$E_{k'}$ with all fibres isomorphic to Kirchberg algebras in the bootstrap category such that $E_k$ and~$E_{k'}$ have isomorphic filtrated $\K$\nb-theory but non-isomorphic filtrated $\K$\nb-theory with coefficients.

  Let~$A$ be a Kirchberg algebra in the bootstrap category with $\K_0(A)=0$ and $\K_1(A)=\Z/n$ for $n\geq 2$.  For $k\in \Z/n$ let \(\varphi_k\colon A \to \Cuntza_\infty\) be a \Star{}homomorphism such that $[\varphi_k]=k \in \KK(A,\Cuntza_\infty) \cong \Z/n$.  Consider the \(\Cont(X)\)-algebra
  \[
  E_k=\{(f,a)\in \Cont(X,\Cuntza_\infty)\oplus A: f(\infty)=\varphi_k(a)\}.
  \]
  We note that \(\K_*(E_k)\cong \K_*(E_{k'})\) as \(\Cont(X,\Z)\)-modules for any \(k,k'\), and we claim that if \(k \Z/n\neq k'\Z/n\), then \(\Kcoef(E_k)\ncong \Kcoef(E_{k'})\) as \(\Cont(X,\Lambda)\)-modules.  Indeed, \(\K_0(E_k)=\K_0(E_{k'})=\Cont_0(X,\Z)\) with \(\Cont(X,\Z)\) acting by pointwise multiplication and \(\K_1(E_k)=\K_1(E_{k'})=\Z/n\) with \(\Cont(X,\Z)\)-module structure \(fm=f(\infty)m\) for \(m\in \Z/n\).  On the other hand,
  \[
  \K_0(E_k;\Z/n)=\{(f,r)\in \Cont(X,\Z/n)\oplus \Z/n: f(\infty)=kr\}.
  \]
  The coefficient map \(\rho\colon \K_0(E_k)\to \K_0(E_k;\Z/n)\) is \(g\mapsto (\dot{g},0)\).  The B\"ockstein map \(\beta\colon \K_0(E_k;\Z/n)\to \K_1(E_k)\) is \(\beta(f,r)=r\).

  Suppose that \(\alpha\colon \Kcoef(E_k)\to \Kcoef(E_{k'})\) is an isomorphism of \(\Cont(X,\Lambda)\)-modules.  Then~\(\alpha\) must act on~$\K_0$ by multiplication by a function $u\colon X \to \{-1,1\}$.  Since~\(\alpha\) is $\Cont(X,\Z)$-linear and commutes with \(\rho\) and~$\beta$, there is a unit $v\in \Z/n$ such that \(\alpha\colon \K_0(E_k;\Z/n)\to \K_0(E_{k'};\Z/n)\) is given by \(\alpha(f,r)=(uf,vr)\).  Choose~$f$ such that $(f,1)\in \K_0(E_k)$.  It follows that for all sufficiently large~$i$ we have $u(i)f(i)=k'v$ and hence $\pm kr=k'v$.  Thus \(k \Z/n= k'\Z/n\).
\end{example}

Next we generalise the previous example, constructing a suitable continuous $\Cont(X)$-algebra over any compact Hausdorff space~\(X\).

\begin{example}
  \label{exa:Filtrated_K_insufficient_general}
  Let~$X$ be an infinite metrisable compact space.  We shall exhibit two unital separable continuous $\Cont(X)$-algebras $F$ and~$F'$ with all fibres isomorphic to Kirchberg algebras in the bootstrap category such that $F$ and~$F'$ have isomorphic filtrated $\K$\nb-theory but non-isomorphic filtrated $\K$\nb-theory with coefficients.

  Using the assumption on~$X$ we find a sequence $(x_i)_{i=1}^\infty$ of distinct elements of~$X$ which converges to some $x_\infty \in X$.  Fix an embedding $\Cuntza_\infty\subset \Cuntza_2$.  For each $k\in \Z/n$ let $A$ and $\varphi_k\colon A \to \Cuntza_\infty$ be as in Example~\ref{exa:Filtrated_K_insufficient}.  Consider the $\Cont(X)$-algebra
  \[
  F_k \defeq \{(f,a)\in \Cont(X,\Cuntza_2)\oplus A \mid
  \text{\(f(x_i)\in \Cuntza_\infty\) for all $i\in \N$, \(f(x_\infty)=\varphi_k(a)\)}\}.
  \]
  Choose $k,k'\in \Z/n$ such that $k \Z/n\neq k' \Z/n$ and set $F=F_k$ and $F'=F_{k'}$.  Then $F$ and~$F'$ have non-isomorphic filtrated $\K$\nb-theory with coefficients since their restrictions to the subspace $Y\defeq\{x_\infty\}\cup \{x_i: i\in \N\}$ are isomorphic to the $\Cont(Y)$-algebras $E_k$ and~$E_{k'}$ from Example~\ref{exa:Filtrated_K_insufficient}, respectively.  At the same time, we have an exact sequence of $\Cont(X)$-algebras \(G \into F_k \prto E_{k}\) with $G=\Cont_0(X\setminus Y, \Cuntza_2)$.  Since $\K_*(\Cuntza_2)=0$, we see that $\K_*\bigl(G(T\setminus Y)\bigr)=0$ for all locally closed subsets~$T$ of~$X$.  It follows that the filtrated $\K$\nb-theory of~$F$ is isomorphic to the filtrated $\K$\nb-theory of~$F'$ since we have seen that $E_k$ and~$E_{k'}$ have this property.
\end{example}

\end{document}